\documentclass[11pt,a4paper]{article}

\usepackage[utf8]{inputenc}
\usepackage[english]{babel}
\usepackage{amsmath,amssymb,amsthm}
\usepackage{thmtools}
\usepackage[table,dvipsnames]{xcolor}
\usepackage{booktabs}
\usepackage{multirow}
\usepackage[scale=0.75]{geometry}
\usepackage[ruled,vlined,linesnumbered]{algorithm2e}
\usepackage{titlesec}
\usepackage{graphicx}
\usepackage{subcaption}
\usepackage{booktabs}

\usepackage[pagebackref,backref=page]{hyperref}
\usepackage{cite}
\renewcommand*{\backrefalt}[4]{%
    \ifcase #1 \footnotesize{(Not cited.)}%
    \or        \footnotesize{(Cited on page~#2)}%
    \else      \footnotesize{(Cited on pages~#2)}%
    \fi}

\titleformat{\subsection}[runin]{\normalfont\bfseries}{\thesubsection.}{0.5em}{}[.]
\titlespacing*{\subsection}{0pt}{1.2ex plus .3ex}{0.6em}

\definecolor{shadecolor}{gray}{0.95}
\declaretheoremstyle[
  headfont=\normalfont\bfseries, 
  notefont=\mdseries, 
  notebraces={(}{)},
  bodyfont=\normalfont, 
  postheadspace=0.5em, 
  spaceabove=8pt, 
  spacebelow=8pt
]{shaded}

\declaretheorem[style=shaded,within=section]{theorem}
\declaretheorem[style=shaded,sibling=theorem]{lemma}
\declaretheorem[style=shaded,sibling=theorem]{proposition}
\declaretheorem[style=shaded,sibling=theorem]{corollary}
\declaretheorem[style=shaded,sibling=theorem]{assumption}
\declaretheorem[style=shaded,sibling=theorem]{definition}
\declaretheorem[style=shaded,sibling=theorem]{remark}

\newcommand{\RR}{\mathbb{R}}
\newcommand{\EE}{\mathbb{E}}
\newcommand{\Eidx}{\mathbb{E}_{\mathcal{B}}}   
\newcommand{\Enoise}{\mathbb{E}_{\varepsilon}} 
\newcommand{\Etot}{\mathbb{E}}                 
\newcommand{\PP}{\mathbb{P}}
\newcommand{\Tr}{\operatorname{Tr}}
\newcommand{\Diag}{\operatorname{Diag}}
\newcommand{\inprod}[2]{\left\langle #1, #2\right\rangle}
\newcommand{\norm}[1]{\lVert #1\rVert}
\DeclareMathOperator*{\argmin}{argmin}
\newcommand{\hatx}{\hat{x}}
\newcommand{\Bcal}{\mathcal{B}}

\newcommand{\Df}[3]{D_f^{#1}(#2,#3)}
\newcommand{\smax}{\sigma_{\max}}
\newcommand{\Cnoise}{C_{\mathrm{noise}}}
\newcommand{\eqdef}{\overset{\text{def}}{=}}

\usepackage[colorinlistoftodos,bordercolor=orange,backgroundcolor=orange!20,linecolor=orange,textsize=scriptsize]{todonotes}

\title{Accelerated Exact Recovery from Noisy Data via Averaging and Noise-Aware Adaptive Bregman–Kaczmarz}

\author{Lionel Tondji\footnote{Corresponding author: \texttt{l.ngoupeyou-tondji@tu-braunschweig.de}}\ \thanks{Institute for Analysis and Algebra, TU Braunschweig, 38092 Braunschweig, Germany and African Institute for Mathematical Sciences (AIMS), \texttt{tngoupeyou@aimsammi.org}.},
  Abakar A. Mahamat\thanks{African Institute for Mathematical Sciences (AIMS) and AMMI Senegal, \texttt{amahamat@aimsammi.org}},
  Idriss Tondji\thanks{University of Pisa, Pisa, Italy, \texttt{idriss.tondji@phd.unipi.it}}}
\date{\today}

\begin{document}
\maketitle

\begin{abstract}
The adaptive Bregman–Kaczmarz method recovers the exact, noise-free solution of a linear inverse problem even when every measurement it queries is corrupted, provided the corruption is fresh, independent and zero-mean. A block version was proposed for parallel hardware, but whether larger blocks actually converge faster was left open. We show the answer hinges on how the block is used: replacing the block sum with a block average collapses the analysis onto a single positive-semidefinite matrix, through which we prove that the guaranteed convergence improves monotonically with the batch size, with a total gain governed by the stable rank of the system matrix. We then address heterogeneous noise by introducing a noise-aware weighting that down-weights unreliable measurements, and prove it is strictly better than uniform weighting whenever the noise is not proportional to the row norms — a condition that essentially never holds in practice. The two improvements are compatible and their benefits combine. Finally, we show the adaptive step size interpolates automatically between a fast initial phase and a slowly vanishing tail that carries the error exactly to zero, and we explain how its hyperparameters can be estimated without knowing the true solution. Numerical experiments under sparse, heterogeneous noise illustrate these findings and confirm that the heuristic estimates produce effective step sizes.
\end{abstract}

\medskip

\section{Introduction}\label{sec:intro}
Many problems in imaging, signal processing and data science reduce to recovering an unknown
$\hatx$ from indirect measurements modeled by a linear map $\mathbf{A}$. The measurements are never
exact: instead of the clean data $b=\mathbf{A}\hatx$ one has access only to a corrupted version, and the
task is to reconstruct $\hatx$ from this corrupted data together with the known
operator~\cite{engl1996regularization,mueller2021}. When $\mathbf{A}$ is large, forming or factorising it
is impractical, and methods that touch only a handful of its rows per step become attractive.
These \emph{row-action} and \emph{block-action} schemes trace back to the Kaczmarz
iteration~\cite{Kac37}, rediscovered in tomography as the algebraic reconstruction
technique~\cite{gordon1970algebraic}, and they remain a workhorse wherever the operator is too
large to store densely.

We study a
generalization of the Kaczmarz method, namely the Bregman–Kaczmarz method. The Bregman-Kaczmarz
method~\cite{lorenz2014linearized,schopfer2019linear} replaces the Euclidean projections of the
classical iteration by Bregman projections with respect to a strongly convex, everywhere-finite
function $f$, and thereby solves the constrained problem
\begin{equation}\label{eq:PB}
  \hatx \;\eqdef\; \argmin_{x\in\RR^n} f(x)\quad\text{subject to}\quad \mathbf{A}x=b.
\end{equation}
The difficulty addressed here is the noise. In the \textit{independent-noise} model of the adaptive Bregman–Kaczmarz (ABK) method~\cite{tondji2024adaptive}, each access to a row returns that entry of $b$ perturbed by a fresh, zero-mean error, yet an adaptively tuned step size still drives the iterates to the \emph{exact} solution $
\hat x$. ABK includes a block variant for parallel hardware, but its analysis leaves open whether larger blocks actually converge faster. We resolve this by \emph{averaging} the row contributions in each block rather than \emph{summing} them and by additionally exploiting knowledge of the per-row noise levels through a non-uniform weighting. We call the method Adaptive Averaged Bregman–Kaczmarz (AABK), see Fig~\ref{fig:aabk-iteration}.

\subsection{Problem statement}\label{sec:problem}
We are given a matrix $\mathbf{A}\in\RR^{m\times n}$ with rows $a_1^\top,\dots,a_m^\top$ and a $1$-strongly
convex function $f:\RR^n\to\RR$ (any strongly convex $f$ may be rescaled to $\mu=1$). Assuming the
system $\mathbf{A}x=b$ is consistent, we seek the minimum-$f$ solution
\[
   \hatx \;=\; \argmin_{x\in\RR^n} f(x)\quad\text{subject to}\quad \mathbf{A}x=b .
\]
We do not observe $b$ directly. Each time the algorithm queries entry $j$ at iteration $k$, it
receives a \emph{fresh} noisy sample
\begin{equation}\label{eq:noise-model}
  \tilde b_j^{(k)} \;=\; b_{i_j} + \varepsilon_j^{(k)},\qquad
  \Enoise\!\big[\varepsilon_j^{(k)}\big]=0,\qquad \Enoise\!\big[(\varepsilon_j^{(k)})^2\big]=\sigma_j^2,
\end{equation}
where the scalar noises $\varepsilon_j^{(k)}$ are independent across both the row index $j$ and
the iteration $k$. We refer to~\eqref{eq:noise-model} as the \textit{independent-noise} model; it is
precisely the freshness of the noise at each query that makes exact recovery of $\hatx$ possible,
rather than mere convergence to a noise ball. We write $\mathbf{\Sigma}=\Diag(\sigma_1^2,\dots,\sigma_m^2)$ and call $\sigma=(\text{Tr}(\mathbf{\Sigma}))^{1/2}$ the total noise level. The strongly convex $f$
resolves non-uniqueness when $n>m$ and a suitable
choice of $f$ shapes the solution: taking $f(x)=\lambda\norm{x}_1+\tfrac12\norm{x}_2^2$ promotes
sparsity, recovering the sparse Kaczmarz method of~\cite{schopfer2019linear}.

\begin{remark}[Scope, failure modes, and relaxations of the fresh-noise model]
\label{rem:noise_model_scope}
The \emph{independent noise} model~\eqref{eq:noise-model} is the core prerequisite enabling exact recovery, as empirical averaging cancels out zero-mean fluctuations ($\mathbb{E}[\varepsilon_j^{(k)}] = 0$). This model accurately captures \emph{streaming physical} measurement environments. A primary motivating example is computerized tomography (CT)~\cite{hansen2021computed}, where repeated rapid exposures can be queried sequentially without dose constraints in industrial non-destructive testing, or in streaming ultra-low-dose dynamic CT and photon-counting imaging where detector readouts exhibit temporally independent Poisson/thermal noise per pulse~\cite{engl1996regularization}; in such settings, reconstruction can run concurrently with acquisition, querying additional projection blocks on the fly. Similar conditions hold for continuous-wave radar, radio-astronomy interferometry, and online acoustic sensing. Conversely, in standard offline problems with a single static corrupted vector $y = \mathbf{A}\hat{x} + e$, repeatedly querying the same data violates conditional unbiasedness ($\mathbb{E}[e_i \mid x_k] = e_i \ne 0$). Under static noise, exact convergence fails, iterates stall at the classical noise ball~\cite{needell2010randomized}, and applying a decaying step size $\eta_k \sim 1/k$ blindly causes premature false convergence. For semi-static data, this limitation can be relaxed via single-pass sampling without replacement~\cite{marshall2023optimal}, randomized surrogate hyperplanes~\cite{dong2025surrogate}, or extended Bregman--Kaczmarz schemes~\cite{dong2026relaxed, schopfer2022extended, zouzias2013randomized} that estimate and remove the deterministic noise component.
\end{remark}

\subsection{Related work}\label{sec:related}
The Kaczmarz literature is large; we organise it around the two ingredients this paper
combines \emph{averaging} a block of row updates and \emph{adapting} the step size to handle
noise---noting for each line what it achieves and what it leaves open.

\emph{Randomisation and the noise ball.} Randomised Kaczmarz is due to Strohmer and
Vershynin~\cite{strohmer2009randomized}, who proved exponential expected convergence under the
sampling rule $p_i\propto\norm{a_i}^2$. Needell~\cite{needell2010randomized} extended it to
inconsistent systems, where the iterates converge to a noise ball whose radius scales with the
residual at the solution, the natural fixed-step picture. Zouzias and
Freris~\cite{zouzias2013randomized} escape the ball via an extended scheme converging to the
least-squares solution, later sharpened for sparse and impulsive-noise problems by Sch\"opfer et al.~\cite{schopfer2022extended}. These works treat noise as a
\emph{fixed} perturbation rather than something the algorithm can average away.

\emph{Bregman geometry.} Sch\"opfer and Lorenz~\cite{schopfer2019linear} introduced the
Bregman--Kaczmarz framework, replacing Euclidean projections by Bregman projections for a
strongly convex $f$; with $f(x)=\lambda\norm{x}_1+\tfrac12\norm{x}_2^2$ one obtains a sparse
solver, and the error-bound condition we use (Assumption~\ref{ass:errorbound}) originates here. Beyond the basic scheme, a broad literature accelerates the Kaczmarz iteration in the noise-free setting through distinct algorithmic mechanisms:
\begin{itemize}
    \item \emph{Greedy and non-uniform sampling:} Bai and Wu~\cite{bai2018greedy} introduced greedy randomised Kaczmarz using residual thresholding to avoid projecting onto equations with negligible residuals. Haddock and Ma~\cite{haddock2021greed} developed the sampling Kaczmarz--Motzkin method, which samples a small candidate subset of rows and greedily executes a Motzkin projection on the row with the maximal residual. Steinerberger~\cite{steinerberger2021weighted} formulated a weighted randomized scheme adjusting sampling probabilities according to pairwise row interactions rather than standard row norms. 
     \item \emph{Sketching and momentum acceleration:} Johnson--Lindenstrauss embeddings was used in~\cite{eldar2011acceleration}  to reduce per-iteration inner-product costs. Liu and Wright~\cite{liu2016accelerated} introduced accelerated randomized Kaczmarz by integrating Nesterov-type momentum, proving faster linear rates for consistent systems. 
     \item \emph{Block subspace projections:} Needell and Tropp~\cite{needell2014paved} analyzed randomized block Kaczmarz based on row pavings, projecting orthogonally via the Moore--Penrose pseudoinverse of row sub-matrices. 
     \item \emph{Weighted, accelerated, and inconsistent Bregman methods:} In the Bregman setting, Tondji et al.~\cite{tondji2023accelerated} studied dual extrapolation acceleration, while Tondji et al.~\cite{tondji2024acceleration} investigated momentum restart mechanisms for ill-conditioned systems. Weighted randomized sparse Kaczmarz algorithms for clean underdetermined linear systems were analyzed by Tondji et al.~\cite{tondji2021linear} and Zhang et al.~\cite{zhang2023weighted}. Convergence theory for randomized Bregman--Kaczmarz methods under batching, averaging, and inconsistent systems was developed in~\cite{tondji2024advances}.
\end{itemize}
All of the aforementioned analyses operate in the noise-free regime or converge only to a residual error floor when measurements are corrupted.

\emph{Block methods and averaging.} Necoara~\cite{necoara2019faster} analysed block Kaczmarz
with extrapolation and rates depending on the block structure of $\mathbf{A}$. Moorman et al.~\cite{moorman2021randomized} introduced randomised Kaczmarz with \emph{averaging},
producing a $1/\tau$ variance reduction; Tondji and Lorenz~\cite{tondji2023faster} carried
averaging to the Bregman setting (their RSKA method) and treated the inconsistent case, where
fixed-step averaging still converges only to a noise ball. Recently, Dong et al.~\cite{dong2026relaxed} developed a relaxed randomized averaging-block extended Bregman--Kaczmarz method for combined optimization problems, establishing variance reduction through block relaxation techniques.

\emph{Adaptive step sizes and independent noise.} A line complementary to the corruption-robust
methods above keeps \emph{all} entries of $b$ noisy but assumes the corruptions are independent
zero-mean random variables. Marshall and Mickelin~\cite{marshall2023optimal} showed, for a fixed
noisy right-hand side $\tilde b$, that one can recover $\hatx$ to any accuracy given enough
equations when rows are sampled \emph{without replacement}: this guarantees that the noise in each
iterate is independent of the noise seen so far. Their guarantee, however, holds for only a single
pass over the matrix at most $m$ iterations so it cannot in practice drive the error to
zero after one epoch. The adaptive Bregman-Kaczmarz in~\cite{tondji2024adaptive} modifies this to the \emph{independent}-noise
model~\eqref{eq:noise-model}, in which every query returns a fresh independent sample, and proves
that an adaptive step size shrinking at the right rate recovers $\hatx$ \emph{exactly} over
arbitrarily many iterations. In a related vein, Dong et al.~\cite{dong2025surrogate} developed a surrogate hyperplane Bregman--Kaczmarz method for linear inverse problems subject to independent noise, using adaptive step sizes along surrogate projection directions. A distinct, corruption-oriented response is the quantile line of work (see~\cite{coria2024quantile,haddock2022quantile,steinerberger2023quantile,zhang2022quantile}), which assumes a small fraction of entries are
adversarially corrupted and rejects rows by residual quantile.

\emph{Where this paper sits.} The two threads that meet our work are averaging
(\cite{tondji2023faster}) and adaptive stepping (\cite{tondji2024adaptive}); they are
complementary, as the table makes explicit.
\begin{center}
\begin{tabular}{lcc}
\toprule
 & fixed step & adaptive step (independent noise)\\
\midrule
no averaging & RK ~\cite{gower2019adaptive,needell2010randomized,strohmer2009randomized}, sparse RK \cite{lorenz2014linearized,schopfer2019linear}
             & ABK \cite{marshall2023optimal,tondji2024adaptive}\\
averaging    & RKA \cite{moorman2021randomized}, RSKA \cite{tondji2023faster}
             & \textbf{AABK (this paper)}\\
\bottomrule
\end{tabular}
\end{center}
\noindent RSKA averages but, in the inconsistent case, converges only to a noise ball; ABK
recovers $\hatx$ exactly but its block bound does not improve with $\tau$, since its update is a
sum rather than an average. Combining the two strengths is not a formality: it requires the spectral
object $\mathbf{T}$ that governs the joint analysis, the proof that $\smax(\mathbf{T})$ decreases monotonically in
$\tau$ (Proposition~\ref{prop:taumono}), and an adaptive step that respects this monotonicity.
Moreover, none of the above works use the noise variances $\sigma_i^2$ to design the sampling
distribution; all default to $p_i\propto\norm{a_i}^2$, optimal for the noiseless geometry but
blind to noise. Our Corollary~\ref{cor:optimal} closes this gap with the rule
$p_i\propto\sigma_i\norm{a_i}$, provably reducing the noise constant by a factor that is at most
$m$ when the row norms are equal and grows without bound only when the most corrupted rows have small geometric influence. In short, AABK enjoys monotone acceleration in
$\tau$ from averaging \emph{and} exact recovery of $\hatx$ from adaptive stepping, in a single
method whose analysis reduces to the matrix $\mathbf{T}$.

\paragraph{Contribution.} Our main contributions are as follows.
\begin{enumerate}
 \item We introduce the AABK method and derive its optimal adaptive step-size rule that guarantees convergence to the noise-free solution
    (Theorem~\ref{thm:main}).
  \item We prove that the convergence rate of AABK is monotonically improving in the batch size
    $\tau$, with an explicit limit and gain governed by the stable rank of $\mathbf{A}$
    (Proposition~\ref{prop:taumono}), giving a clear theoretical justification for larger batches.
  \item We develop a noise-aware non-uniform weighting scheme and prove its theoretical
    superiority: weights proportional to $\|a_i\|/\sigma_i$ yield a provably smaller rate constant (Corollary~\ref{cor:optimal}).
  \item We analyse the asymptotic behaviour of the adaptive step size, showing the automatic
    transition from a constant-step (linear) regime to a decaying-step (stochastic) regime, and
    we discuss how the hyperparameters are estimated in practice (\S\ref{sec:asymp}--\ref{sec:hyper}).
\end{enumerate}

\section{Results}\label{sec:results}

Let define the diagonal matrices :
\[
  \mathbf{D}=\Diag(\norm{a_1},\dots,\norm{a_m}),\quad
  \mathbf{P}=\Diag(p_1,\dots,p_m),\quad
  \mathbf{W}=\Diag(w_1,\dots,w_m),
\]
where $p_i$ is the probability of sampling row $i$ and $w_i\ge0$ its weight. Following common practice in the randomized Kaczmarz literature~\cite{moorman2021randomized,strohmer2009randomized,tondji2023faster}, we impose throughout
the \emph{coupling}
\begin{equation}\label{eq:coupling}
  \frac{p_i w_i}{\norm{a_i}^2}=\frac{\alpha}{\norm{\mathbf{A}}_F^2}\quad\text{for all }i,
  \qquad\text{equivalently}\qquad \mathbf{PWD}^{-2}=\frac{\alpha}{\norm{\mathbf{A}}_F^2}\,\mathbf{I},
\end{equation}
with relaxation parameter $\alpha>0$. The coupling ties the sampling distribution to the weights
through the single scalar $\alpha$; under it, the inner-product term in the analysis collapses to
a clean Rayleigh quotient. Given the coupling, choosing the $w_i$ determines the $p_i$ and vice
versa, and one of the two natural choices is provably better (\S\ref{sec:weights}).

\subsection{The algorithm}\label{sec:algorithm}
We were motivated by the following heuristic. Suppose $x_k$ is given and we query the $i$-th
equation, but its noise variance $\sigma_i^2$ is large. Then we should not take a full step that
exactly satisfies this equation: the measurement is mostly noise, and trusting it blindly would
undo the progress of earlier iterations. We should actively suppress equations whose incoming
data $\tilde b_i$ are known to be unreliable. Moreover, stepping on a single noisy equation
prevents the iterate from ever settling. A natural remedy is therefore to query a batch of $\tau$
equations simultaneously, average their proposed updates so that the zero-mean noise partly
cancels, and weight each equation inversely to its unreliability. Uniform averaging of updates is
classical, it underlies mini-batch stochastic gradient descent and has been analysed for
Kaczmarz methods in~\cite{moorman2021randomized,tondji2023faster}. We adopt this idea in the
Bregman setting and add a dynamically shrinking step size $\eta_k$ that drives the remaining
averaged error exactly to zero.

Concretely, at each iteration $k$ we sample a batch of indices $\Bcal_k\subseteq\{1,\dots,m\}$ of
size $\tau$, i.i.d.\ from $p$ (with replacement), and form the weighted average of the individual
Kaczmarz directions,
\begin{equation}\label{eq:dk}
  d_k=\frac1\tau\sum_{j\in\Bcal_k} w_i\,
      \frac{\inprod{a_{i_j}}{x_k}-\tilde b_j^{(k)}}{\norm{a_{i_j}}^2}\,a_{i_j},
\end{equation}
where $\tilde b_j^{(k)}$ is a fresh noisy measurement. We then update the dual and primal
variables via
\begin{equation}\label{eq:update}
  x_{k+1}^*=x_k^*-\eta_k d_k,\qquad x_{k+1}=\nabla f^*(x_{k+1}^*).
\end{equation}

\begin{figure}[t]
  \centering
  \includegraphics[width=\linewidth]{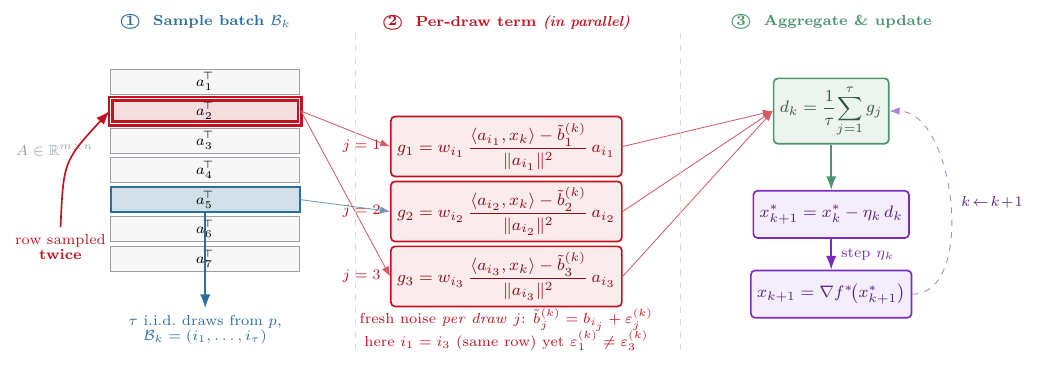}
  \caption[One iteration of AABK]{\textbf{Anatomy of one AABK iteration.}
  Starting from $x_k=\nabla f^*(x_k^*)$, the method
  \textcircled{\scriptsize 1} samples a batch $\Bcal_k=(i_1,\dots,i_\tau)$ of $\tau$
  rows i.i.d.\ from $p$ (with replacement, so a row may recur---here $i_1=i_3$);
  \textcircled{\scriptsize 2} computes in parallel one residual term
  $g_j=w_{i_j}\frac{\inprod{a_{i_j}}{x_k}-\tilde b_j^{(k)}}{\norm{a_{i_j}}^2}a_{i_j}$
  per draw, each using a \emph{fresh} noisy value
  $\tilde b_j^{(k)}=b_{i_j}+\varepsilon_j^{(k)}$ (indexing noise by the draw $j$
  guarantees independence even when a row is drawn twice); and
  \textcircled{\scriptsize 3} averages them into $d_k=\frac1\tau\sum_{j=1}^{\tau}g_j$,
  takes the adaptive dual step $x_{k+1}^*=x_k^*-\eta_k d_k$, and mirrors back via
  $x_{k+1}=\nabla f^*(x_{k+1}^*)$. The dashed arrow closes the loop for iteration $k{+}1$.}
  \label{fig:aabk-iteration}
\end{figure}

The method rests on three pillars: the batch size $\tau$ to reduce variance, the weights $w_i$ to
handle heteroskedasticity, and the adaptive step $\eta_k$ to guarantee exact convergence. For
$\tau=1$ and uniform weights it reduces to ABK~\cite{tondji2024adaptive}; in the noiseless case
with uniform weights it reduces to RSKA~\cite{tondji2023faster,tondji2024advances}. Computing $d_k$ is a sparse
matrix--vector product(\S\ref{sec:matvec}) and is embarrassingly parallel across the batch.

\begin{algorithm}[H]
\caption{Adaptive Averaged Bregman--Kaczmarz (AABK)}\label{alg:aabk}
\KwIn{$\mathbf{A}\in\RR^{m\times n}$; strongly convex $f$; weights $\{w_i\}$ and probabilities
      $\{p_i\}$ satisfying~\eqref{eq:coupling}; batch size $\tau$; initial $x_0^*$,
      $x_0=\nabla f^*(x_0^*)$.}
\KwOut{Approximate minimum-$f$ solution of $\mathbf{A}x=b$.}
\For{$k=0,1,2,\dots$ until a stopping criterion holds}{
  Sample a batch $\Bcal_k = \{i_1, \dots,i_{\tau}\}\subseteq\{1,\dots,m\}$, $|\Bcal_k|=\tau$, i.i.d.\ from $p$\;
  For each $j\in\Bcal_k$ query a fresh noisy value $\tilde b_j^{(k)}=b_{i_j}+\varepsilon_j^{(k)}$\;
  $d_k\leftarrow \dfrac1\tau\sum_{j\in\Bcal_k} w_{i_j}\dfrac{\inprod{a_{i_j}}{x_k}-\tilde b_j^{(k)}}
        {\norm{a_{i_j}}^2}\,a_{i_j}$\tcp*{in parallel}
  Set $\eta_k$ by~\eqref{eq:etaopt}\;
  $x_{k+1}^*\leftarrow x_k^*-\eta_k d_k$;\quad $x_{k+1}\leftarrow\nabla f^*(x_{k+1}^*)$\;
}
\Return $x_k$\;
\end{algorithm}

\subsection{Main convergence theorem}\label{sec:mainthm}
The convergence is controlled by a single symmetric, positive semidefinite matrix
\begin{equation}\label{eq:Tdef}
  \mathbf{T}=\frac{1}{2\tau}\,\mathbf{W}+\frac{\alpha}{2\norm{\mathbf{A}}_F^2}\Big(1-\frac1\tau\Big)\mathbf{AA}^\top .
\end{equation}

\begin{theorem}[Convergence of AABK]\label{thm:main}
Assume the coupling~\eqref{eq:coupling} and Assumption~\ref{ass:errorbound}. With $\mathbf{T}$ as
in~\eqref{eq:Tdef}, define the adaptive step size and auxiliary sequence
\begin{equation}\label{eq:etaopt}
  \eta_k=\frac{\alpha\gamma\beta_k}{1+2\alpha\gamma\,\smax(\mathbf{T})\,\beta_k},\qquad
  \beta_{k+1}=\beta_k\Big(1-\tfrac{\alpha\gamma\eta_k}{2}\Big),
\end{equation}
initialized at $\beta_0=\tau\,\Df{x_0^*}{x_0}{\hatx}/\Tr(\mathbf{PW}^2\mathbf{\Sigma D}^{-2})$ (estimated as in
\S\ref{sec:hyper}). If $(x_k)$ is generated by~\eqref{eq:dk}--\eqref{eq:update}, then
\begin{equation}\label{eq:main-rate}
  \Etot\,[\norm{x_k-\hatx}_2^2]\;\le\;\frac{2\,\Tr(\mathbf{PW}^2\mathbf{\Sigma D}^{-2})}{\tau}\,\beta_k,
\end{equation}
where $\beta_k$ is deterministic, strictly decreasing and converges to zero; in particular
$$\Etot\,[\norm{x_k-\hatx}_2^2]\to0.$$
\end{theorem}

\paragraph{Interpretation of Theorem~\ref{thm:main}.}
The bound~\eqref{eq:main-rate} factors the expected error into two pieces,
\[
  \Etot\,[\norm{x_k-\hatx}_2^2] \le 
  \underbrace{\frac{2\,\Tr(\mathbf{PW}^2\mathbf{\Sigma D}^{-2})}{\tau}}_{\text{rate constant  }=\,2\Cnoise/\tau}
  \cdot \underbrace{\beta_k}_{\text{monotone decay}},
\]
each controlled by a different design choice. The prefactor is the noise injected per step; it
depends on $\mathbf{W}$, $\mathbf{P}$, $\mathbf{\Sigma}$ and $\tau$ but not on $k$, and is what the
noise-aware weighting of \S\ref{sec:weights} minimises. While the recurrence update rule from $\beta_k$ to $\beta_{k+1}$ depends only on the iteration index $k$ and the spectral bound $\sigma_{\max}(\mathbf{T})$, the magnitude of $\beta_k$ directly inherits a dependence on both the batch size $\tau$ and the measurement noise profile $\mathbf{\Sigma}$ through the initialization $\beta_0$, and is what the batch size accelerates
(Proposition~\ref{prop:taumono}). Three consequences follow. \emph{(i)~Exact recovery:} since
$\beta_k\to0$, the iterates reach $\hatx$ exactly despite fresh noise at every step,
unlike fixed-step averaging, which stalls at a noise ball~\cite{needell2010randomized,tondji2023faster}.
\emph{(ii)~Composable knobs:} $\tau$ primarily drives the decay factor $\beta_k$ while $\mathbf{W}$ primarily drives the noise prefactor $\Cnoise$. Under the coupling~\eqref{eq:coupling} both depend on $\mathbf{W}$ and $\tau$, so they are not strictly independent; but because the noise-optimal weights also satisfy the batch-monotonicity condition (Proposition~\ref{prop:taumono}), the two gains compose without conflict.
\emph{(iii)~Monotone in $\tau$:} both factors
improve as $\tau$ grows, so the bound is non-increasing in $\tau$ at every fixed $k$; the
justification for block parallelism that the original ABK analysis left open.

\subsection{Larger batches help: monotonicity in \texorpdfstring{$\tau$}{tau}}\label{sec:tau}
\begin{proposition}[Monotone acceleration]\label{prop:taumono}
The following hold.
\begin{enumerate}
  \item \emph{(Uniform weights.)} For $\mathbf{W}=\alpha \mathbf{I}$,
    \[
      \smax(\mathbf{T})=\frac{\alpha}{2\tau}\Big(1+\frac{(\tau-1)\,\smax^2(\mathbf{A})}{\norm{\mathbf{A}}_F^2}\Big),
    \]
    which is strictly decreasing in $\tau$, with $\displaystyle\lim_{\tau\to\infty}\smax(\mathbf{T})
    =\frac{\alpha\,\smax^2(\mathbf{A})}{2\norm{\mathbf{A}}_F^2}$. The contraction factor at the optimal step thus
    improves monotonically in $\tau$, and the improvement from $\tau=1$ to $\tau\to\infty$ equals
    the stable rank $\norm{\mathbf{A}}_F^2/\smax^2(\mathbf{A})$.
    
    \item (General weights.) For arbitrary $\mathbf{W}$, the guaranteed rate obtained by running~\eqref{eq:etaopt} with the Lemma~\ref{lem:Tbound} upper bound $U(\tau) := \frac{1}{2\tau}\,\smax(\mathbf{W})
     +\frac{\alpha}{2\norm{\mathbf{A}}_F^2}\Big(1-\frac1\tau\Big)\smax^2(\mathbf{A})$ in place of $\smax(\mathbf{T})$ improves monotonically in $\tau$ whenever $\smax(\mathbf{W})\ge\alpha\,\smax^2(\mathbf{A})/\norm{\mathbf{A}}_F^2$ which is legitimate, since $\eta_k$ can be run with $U$ in place of $\smax(\mathbf{T})$ and the contraction bound $(1-\alpha\gamma/4U)$ then improves in $\tau$. Because $\mathbf{W}$ and $\mathbf{AA}^{\top}$ need not commute, the true $\smax(\mathbf{T})$ itself may fail to be monotone for general $\mathbf{W}$; the guarantee is on the certified bound $U(\tau)$. 
  \item \emph{(Optimal noise-aware weights.)} For the optimal weights of
    Corollary~\ref{cor:optimal}, the condition in (ii) holds automatically: in fact
    $\smax(\mathbf{W})\ge\alpha\ge\alpha\,\smax^2(\mathbf{A})/\norm{\mathbf{A}}_F^2$, for every matrix $\mathbf{A}$ and every noise
    profile $(\sigma_i)$. Consequently the rate of AABK with optimal weights (via the bound $U(\tau)$ of part (ii))  accelerates monotonically in
    $\tau$ unconditionally, so the monotone-acceleration and optimal-weighting guarantees compose
    with no compatibility condition.
  \end{enumerate}
\end{proposition}
\noindent When $\tau=1$, $\smax(\mathbf{T})=\alpha/2$ and we recover the single-row case. When
$\tau\to\infty$, $\smax(\mathbf{T})$ shrinks by a factor equal to the stable rank of $\mathbf{A}$: for
matrices with a flat spectrum this can be as large as $\min(m,n)$, so averaging is dramatic; for
matrices dominated by one singular value it is close to $1$ and averaging buys little. The
unaveraged block bound of~\cite{tondji2024adaptive} does not exhibit this monotonicity precisely
because its update sums rather than averages, so the variance-reduction term $1/\tau$ is absent.
Part~(iii) is worth emphasising: although the condition in part~(ii) looks like a hypothesis the
user must verify, the optimal scheme satisfies it with room to spare by a factor of at least
the stable rank because it places its largest weight on the row that is most reliable
relative to its geometric influence. The proof of part~(iii) is given in Appendix~\ref{app:optmono}.

\subsection{Uniform versus noise-aware weighting}\label{sec:weights}
Uniform weights treat all rows equally; the noise-aware scheme down-weights unreliable rows. The
two are compared through the noise prefactor $\Cnoise:=\Tr(\mathbf{PW}^2\mathbf{\Sigma D}^{-2})$.

\begin{corollary}[Uniform versus noise-aware weighting]
\label{cor:optimal}
\quad\quad
\begin{enumerate}
    \item Uniform weights : with $\mathbf{W}=\alpha \mathbf{I}$ and $p_i=\norm{a_i}^2/\norm{\mathbf{A}}_F^2$,
\[
  \Tr(\mathbf{PW}^2\mathbf{\Sigma D}^{-2})=\frac{\alpha^2}{\norm{\mathbf{A}}_F^2}\sum_{i=1}^m\sigma_i^2 .
\]
A single row with large $\sigma_i^2$ inflates the whole sum: the uniform scheme is brittle to
outliers.
     
    \item Optimal noise-aware weights : choosing $w_i\propto \norm{a_i}/\sigma_i$ and $p_i\propto\sigma_i\norm{a_i}$ (preserving the
coupling~\eqref{eq:coupling}) gives
\[
  \Tr(\mathbf{PW}^2\mathbf{\Sigma D}^{-2})=\frac{\alpha^2}{\norm{\mathbf{A}}_F^4}\Big(\sum_{i=1}^m\sigma_i\norm{a_i}\Big)^2,
\]
which is the minimum over all couplings. By the Cauchy--Schwarz inequality, equality holds if and only if the noise standard deviations are strictly proportional to the row norms, i.e., $\sigma_i \propto \|a_i\|$ for all $i=1,\dots,m$. In the special case of homoscedastic noise ($\sigma_i = \sigma$ for all $i$), equality holds if and only if the row norms are identical ($\|a_i\| = \|a_j\|$ for all $i,j$).
\end{enumerate}
\end{corollary}
\noindent When the noise per-row are not equal, equality holds only when the per-row noise is proportional to the row norm. In practice
this essentially never occurs; noise is a property of the measurement device while the row norm
is a property of the operator so the strict improvement is generic. 

\subsection{Asymptotics of the step size and error}\label{sec:asymp}
Since the decay of $\beta_k$
governs how fast Algorithm~\ref{alg:aabk} converges to the noise-free solution
$\hatx$, we now quantify this decay. Let us denote by $h(k) = \alpha\Tr(\mathbf{W\Sigma}) \cdot \beta_k/ \tau$ so that by~\eqref{eq:main-rate},
$\Etot\,[\norm{x_k-\hatx}_2^2]\le 2 h(k)/\|\mathbf{A}\|_F^2$. The following result generalizes the one from~\cite{marshall2023optimal,tondji2024adaptive} with full proof in Appendix~\ref{sec:app-cor-asymp}.
\begin{corollary}[Asymptotic regimes]\label{cor:asymp}
Solving the recursion~\eqref{eq:etaopt} via the Lambert-$W$ function yields the upper bound 
\begin{align*}
h(k) \le  \dfrac{\Tr(\mathbf{W\Sigma})}{2\tau\gamma\sigma_{\max}(\mathbf{T})W_{0}\left(c_0\exp\left(\dfrac{\alpha\gamma k}{4\sigma_{\max}(\mathbf{T})}\right)\right)},
\end{align*}
with $c_0 = \dfrac{1}{2\alpha\gamma\sigma_{\max}(\mathbf{T})\beta_0}\exp\left(\dfrac{1}{2\alpha\gamma\sigma_{\max}(\mathbf{T})\beta_0}\right)$ and $W_{0}(\cdot)$ denoting the principal branch of the Lambert-$W$ function.
Consequently:
\begin{enumerate}
  \item \emph{(Vanishing noise, $\sigma_i\to0\,\,\forall\,i$.)} With all other parameters fixed,
    \begin{align*}
      h(k)  \leq 
            \begin{cases} \norm{\mathbf{A}}_F^2\,
            \exp\!\Big(-\frac{\alpha\gamma k}{4\smax(\mathbf{T})}\Big)\Df{x_0^*}{x_0}{\hatx} + \mathcal{O}\big(\sigma^2\big),\quad \text{If}\,\, \mathbf{W} = \alpha \mathbf{I}\\\\
           \norm{\mathbf{A}}_F^2\,
            \exp\!\Big(-\frac{\alpha\gamma k}{4\smax(\mathbf{T})}\Big)\Df{x_0^*}{x_0}{\hatx} + \mathcal{O}\big(\bar\sigma^2\big),\quad \text{Else}
            \end{cases},
    \end{align*}
    
    where $\bar\sigma^2=\left(\sum_j \sigma_j \|a_j\|\right)^2$ and the step size satisfies $\eta_k\to 1/(2\smax(\mathbf{T}))$, yielding linear convergence in the noiseless limit.
  \item \emph{(Late iterations, $k\to\infty$.)} With all other parameters fixed,
    \begin{align*}
      h(k)  \leq \frac{2\text{Tr}(\mathbf{W}\mathbf{\Sigma }) }{\alpha \tau\gamma^2 k}\left(1+\mathcal{O}\left(\dfrac{ln(k)}{k}\right)\right),\qquad
      \eta_k\approx\frac{2}{\alpha\gamma k+4\smax(\mathbf{T})},
    \end{align*}
    so that $\eta_k\sim 1/k$ (a Robbins--Monro tail) and
    $\Etot\,[\norm{x_k-\hatx}_2]\leq \frac{2}{\gamma \alpha \sqrt{\tau}} \cdot \sqrt{\frac{\text{Tr}(\mathbf{PW}^2\mathbf{\Sigma D}^{-2})}{k}}.$
\end{enumerate}
\end{corollary}

\begin{remark}[Two regimes of the step size]\label{rem:regimes}
The adaptive step interpolates automatically between two regimes. \emph{Early on}, $\beta_k$
is large and $\eta_k\approx 1/(2\smax(\mathbf{T}))$ the largest admissible step so the error decays like a deterministic linear iteration driven by a near-constant step size. Because $\smax(\mathbf{T})$ shrinks as the batch grows (Proposition~\ref{prop:taumono}), this step can exceed 2: where ABK is confined to $\eta_k\le 1$ and hence under-relaxes, AABK is free to \emph{over-relax}. To the best of our knowledge, this is the first time over-relaxation is shown to help in the noisy setting. \emph{Late on}, $\beta_k$ is small and $\eta_k\approx\alpha\gamma\beta_k\to 0$, decaying as $\eta_k\sim 1/k$ a Robbins-Monro tail. It is this vanishing step that lets AABK absorb the residual noise and converge to the exact solution $\hatx$, rather than settling at a noise ball.
\end{remark}

\subsection{Heuristic estimation of the hyperparameters}\label{sec:hyper}
The step size~\eqref{eq:etaopt} requires $\gamma$ and $\beta_0$, both of which involve unknown
quantities. We estimate them from one auxiliary run of Algorithm~\ref{alg:aabk} using constant step
$\eta_k=1/2\smax(\mathbf{T})$, producing iterates $x_0,\dots,x_N$. Under uniform weights  with $\tau=1$ and $\alpha=1$ one has $\smax(\mathbf{T})=1/2$, so this auxiliary step reduces to $\eta_k=1$, the value used in~\cite{tondji2024adaptive} for their block version. The derivation in this part follows closely the one in~\cite{tondji2024adaptive} and generalizes their results. From the one-step bound with $\eta_k=1/2\smax(\mathbf{T})$,
\begin{align*}
\mathbb{E}[D_f^{x_{k+1}^*}(x_{k+1}, \hat{x})] &\le \left( 1 - \dfrac{\alpha\gamma}{4\smax(\mathbf{T})}\right)\mathbb{E}[D_f^{x_k^*} (x_k, \hat{x})] + \dfrac{\Cnoise}{8\tau\smax^2(\mathbf{T})},
\end{align*}
which can be inferred inductively as follows:
\begin{align*}
\mathbb{E}[D_f^{x_{k}^*}(x_{k}, \hat{x})] &\le \left( 1 -\dfrac{\alpha\gamma}{4\smax(\mathbf{T})}\right)^{k}\mathbb{E}[D_f^{x_0^*} (x_0, \hat{x})] + \dfrac{\Cnoise}{2\tau \alpha \gamma\smax(\mathbf{T})}.
\end{align*}
Assuming the initial error is above the noise level and eventually stagnates, we expect
geometric decay early and a plateau late. Using the approximation that $x_N \approx \hat x$ and choosing an index $N_0<N$ for which we assume $1\le j\le N_0,$ that  $D_f^{x_{j+1}^*}(x_{j+1}, x_\text{N}) \approx \left( 1 -\dfrac{\alpha\gamma}{4\smax(\mathbf{T})}\right)D_f^{x_j^*} (x_j, x_\text{N})$ holds, gives the estimator

\begin{align}\label{gamma-tilde}
\tilde\gamma:=\frac{4\smax(\mathbf{T})}{\alpha}\Bigg(1-\frac{1}{N_0}\sum_{j=1}^{N_0}
      \frac{\Df{x_j^*}{x_j}{x_N}}{\Df{x_{j-1}^*}{x_{j-1}}{x_N}}\Bigg),    
\end{align}
and choosing $N_1< N$ in the plateau phase for which we assume that $N_1\le j<N$, we then have the following
$$D_f^{x_j^*} (x_j, \hat{x}) \approx \frac{\text{Tr}(\mathbf{PW}^2\mathbf{\Sigma D}^{-2}) }{2\tau\gamma \alpha \sigma_{\max}(\mathbf{T})},$$
giving

\begin{align}\label{beta_0_tilde}
\tilde\beta_0:=\frac{1}{2\alpha\smax(\mathbf{T})}
     \Bigg(\frac{\tilde\gamma}{N_1}\sum_{j=N-N_1}^{N-1}
      \frac{\Df{x_j^*}{x_j}{x_N}}{\Df{x_0^*}{x_0}{x_N}}\Bigg)^{-1}.    
\end{align}

\noindent In practice, the method is robust to an order-of-magnitude error in these estimates.
\begin{remark}[Selection rules for window indices $(N_0, N_1)$] From the auxiliary run with constant step $\eta_k = 1/2\smax(\mathbf{T})$, the error curve exhibits a two-phase trajectory: an initial geometric linear decay followed by an asymptotic noise plateau. The index $N_0$ must be selected within the steep linear descent phase (typically $5\%\text{--}15\%$ of the total auxiliary iterations $N$), while $N_1$ must be chosen well inside the flat stagnation plateau (the final $20\%\text{--}30\%$ of the auxiliary run).
\end{remark}

\section{Basic notions}
\label{sec:basicnotions}

We will analyze the convergence of the Algorithm~\ref{alg:aabk} with the help of the Bregman distance with respect to the objective function $f$. To this end,
we recall some well-known concepts and properties of convex functions~\cite{rockafellar1970}.

\noindent Let $f:\RR^n \to \RR$ be convex, and since it is finite everywhere, it is also continuous. The \emph{subdifferential} of $f$ at any $x \in \RR^n$ is defined by
\[
\partial f(x) \eqdef \{x^* \in \RR^n| f(y) \ge f(x) + \langle x^*, y-x \rangle, \forall\, y \in \RR^n \},
\]
which is nonempty, compact, and convex. The function  $f$ is said to be \emph{$\mu$-strongly convex}, if for all $x,y \in \RR^n$ and subgradients $x^* \in \partial f(x)$ we have
\[
f(y) \geq f(x) + \langle x^*, y-x \rangle + \tfrac{\mu}{2} \cdot \|y-x\|_2^2 \,.
\]
If $f$ is $\mu$-strongly convex, then $f$ is coercive, i.e. $\lim_{\|x\|_2 \to \infty} f(x)=\infty$,
and its \emph{Fenchel conjugate} $f^{*}:\RR^n \to \RR$ given by 
\[
f^*(x^*)\eqdef \sup_{y \in \RR^n} \langle x^*,y \rangle - f(y),
\]
is also convex, finite everywhere and coercive. Additionally, $f^*$ is differentiable with a \emph{Lipschitz-continuous gradient} with constant $L_{f^*}=\frac{1}{\mu}$, i.e. for all $x^*,y^* \in \RR^n$ we have
\[
\|\nabla f^*(x^*)-\nabla f^*(y^*)\|_2 \le L_{f^*} \cdot \|x^*-y^*\|_2 \,,
\]
which implies the estimate
\begin{align}\label{eq:Lip}
f^*(y^*)
\le f^*(x^*) +\langle \nabla f^*(x^*), y^*-x^* \rangle + \tfrac{L_{f^*}}{2}\|x^{*}-y^{*}\|_2^2. 
\end{align}

\begin{definition} \label{def:D}
The \emph{Bregman distance} $D_f^{x^*}(x,y)$ between $x,y \in \RR^n$ with respect to $f$ and a subgradient $x^* \in \partial f(x)$ is defined as
\[
D_f^{x^*}(x,y) \eqdef f(y)-f(x) -\langle x^*,y - x \rangle\,.
\]
\end{definition}
Fenchel's equality states that $f(x) + f^*(x^*) = \langle x, x^*\rangle$ if $x^*\in\partial f(x)$ and implies that the Bregman distance can be written as
\[
D_f^{x^*}(x,y) = f^*(x^*)-\langle x^*,y\rangle + f(y)\,.
\]
We will use the function
\begin{equation} \label{eq:spf}
f(x) \eqdef \lambda\|x\|_1 + \tfrac{1}{2}\|x\|_2^{2},
\end{equation}
which is strongly convex with constant $\mu=1$. Its conjugate function is the soft shrinkage operator $S_{\lambda}$ which is defined componentwise by $(S_{\lambda}(x))_j = \max\{|x_j|-\lambda,0\} \cdot \text{sign}(x_j)$
\[ 
f^{*}(x^{*}) = \tfrac{1}{2}\|S_{\lambda}(x^{*})\|_2^{2}, \quad \mbox{with} \quad \nabla f^{*}(x^{*}) = S_{\lambda}(x^{*}) \,.
\]
In general, for a $\mu$-strongly convex function, we have that 
\begin{align} 
\label{eq:D}
  D_f^{x^*}(x,y) \geq \frac{\mu}{2} \|x-y\|_2^2, 
\end{align}
see~\cite{lorenz2014linearized}, and hence $D_f^{x^*}(x,y) = 0 \Leftrightarrow x=y.$

We now prove convergence of the $x_k$ in Algorithm~\ref{alg:aabk} to the solution $\hatx$ for our optimal step size sequence.
To obtain convergence rates for the solution algorithm, we need to assume an error bound which relates the least-squares residual $\|\mathbf{A}x - b\|^2_2$ to the Bregman distance to the exact solution $\hatx$ for iterates of Algorithm~\ref{alg:aabk}.

\begin{assumption}[Error bound]\label{ass:errorbound}
Let  $\mathcal{R}(\mathbf{A})$ denote the range space of $\mathbf{A}$. There  exists a constant $\theta(\hatx)>0$ such that, for each $x \in \mathbb{R}^n$ with $\partial f(x) \cap \mathcal{R}(\mathbf{A}^{\top}) \neq \emptyset$ and $x^* \in \partial f(x) \cap \mathcal{R}(\mathbf{A}^{\top})$ we have : 
\[
\norm{Ax-b}_2^2\ge\theta(\hatx)\,\Df{x^*}{x}{\hatx}.
\] 
We set $\gamma:=\theta(\hatx)/\norm{\mathbf{A}}_F^2$.
\end{assumption}
For the function $f(x)=\lambda\norm{x}_1+\tfrac12\norm{x}_2^2$, it holds that $\text{dom}(f) = \mathbb{R}^n$ and
Assumption~\eqref{ass:errorbound} holds with an explicit constant $\theta$ depending on $\mathbf{A}$ and $\hatx$, see~\cite[Lemma 3.1]{schopfer2019linear}. See also~\cite[Theorem 3.9]{schopfer2022extended} for a sufficient condition for Assumption~\eqref{ass:errorbound} to hold. 
\section{Proofs}\label{sec:proofs}
Throughout this section, we drop the iteration superscript $(k)$ for readability, writing $\tilde b_j=b_{i_j}+\varepsilon_j$ rather than $\tilde b_j^{(k)}=b_{i_j}+\varepsilon_j^{(k)}$. Every expectation is conditional on $x_k$, so $k$ is fixed and understood; the abbreviation is purely notational, the noise remains independent across iterations and its omission changes none of the proofs.

At each iteration $k$, the algorithm draws two independent sources of randomness: the batch
$\Bcal_k$ of $\tau$ indices sampled i.i.d.\ from $p$, and the fresh noise
$\{\varepsilon_i^{(k)}\}$ attached to the queried entries. We distinguish three expectations and
use them consistently throughout:
\begin{itemize}
  \item $\Eidx[\,\cdot\,]$ --- expectation over the batch indices $\Bcal_k$ only, with the
    iterate $x_k$ and the noise held fixed; when $\tau =1$ this reduces to the single-draw expectation $\mathbb{E}_{i\sim p}$.
  \item $\Enoise[\,\cdot\,]$ --- expectation over the streaming noise
    $\{\varepsilon_i^{(k)}\}$ only, with $x_k$ and the batch $\Bcal_k$ held fixed;
  \item $\Etot[\,\cdot\,]=\EE[\,\cdot\,]$ --- the total expectation over all randomness in the
    history (all batches and all noise up to and including iteration $k$).
\end{itemize}
 
The next lemma will be used to prove convergence of Algorithm~\ref{alg:aabk} for suitable step size.

\subsection{One-step descent}\label{sec:onestep}
\begin{lemma}\label{lem:onestep}
Let $(x_k,x_k^*)$ be generated by Algorithm~\ref{alg:aabk}. Then
\begin{align*}
  \Df{x_{k+1}^*}{x_{k+1}}{\hatx}\le \Df{x_k^*}{x_k}{\hatx}
     +\eta_k\inprod{d_k}{\hatx-x_k}+\frac{\eta_k^2}{2}\norm{d_k}_2^2 .
\end{align*}
\end{lemma}
\begin{proof}
For a $1$-strongly convex $f$, the properties of the Bregman distance give
\begin{align*}
  \Df{x_{k+1}^*}{x_{k+1}}{\hatx}\le \Df{x_k^*}{x_k}{\hatx}
     -\inprod{x_{k+1}^*-x_k^*}{\hatx-x_{k+1}}-\tfrac12\norm{x_{k+1}-x_k}_2^2 .
\end{align*}
Substituting $x_{k+1}^*=x_k^*-\eta_k d_k$ and writing
\begin{align*}
-\inprod{x_{k+1}^*-x_k^*}{\hatx-x_{k+1}}=\eta_k\inprod{d_k}{\hatx-x_k}
-\eta_k\inprod{d_k}{x_{k+1}-x_k},
\end{align*}
then bounding the last two terms with
$-az-\tfrac12 z^2\le\tfrac12 a^2$ yields the claim.
\end{proof}

\subsection{The expected one-step bound}\label{sec:expbound}
\begin{lemma}\label{lem:descent}
Under the coupling~\eqref{eq:coupling} and Assumption~\ref{ass:errorbound}, with $\mathbf{T}$ as
in~\eqref{eq:Tdef}, the total expectation satisfies
\begin{align*}
  \Etot\,[\Df{x_{k+1}^*}{x_{k+1}}{\hatx}]\le
     \big(1-\gamma\alpha\eta_k(1-\eta_k\smax(\mathbf{T}))\big)\,\Etot\,[\Df{x_k^*}{x_k}{\hatx}]
     +\frac{\eta_k^2}{2\tau}\,\Tr(\mathbf{PW}^2\mathbf{\Sigma D}^{-2}).
\end{align*}
\end{lemma}
\begin{proof}
We first compute the conditional expectation $\EE[\,\cdot\mid x_k]=\Eidx\Enoise[\,\cdot\mid x_k]$
of Lemma~\ref{lem:onestep} over the batch $\Bcal_k$ and the noise.

\emph{Inner-product term.} Writing $r_k=\mathbf{A}x_k-b$ and using $\inprod{a_i}{\hatx}=b_i$, the noise
enters linearly and is annihilated by $\Enoise[\varepsilon_i]=0$; the surviving index-expectation
is
\begin{align*}
  \EE\big[\inprod{d_k}{\hatx-x_k}\mid x_k\big]
   =-\Eidx\Big[w_i\frac{(r_k)_i^2}{\norm{a_i}^2}\Big]
   &=-\sum_{i=1}^m\frac{p_iw_i}{\norm{a_i}^2}(r_k)_i^2 \\
   &=-(\mathbf{A}x_k-b)^\top \mathbf{PWD}^{-2}(\mathbf{A}x_k-b) \\
   &=-\frac{\alpha}{\norm{\mathbf{A}}_F^2}\norm{\mathbf{A}x_k-b}_2^2,
\end{align*}
the last step by~\eqref{eq:coupling}.

\emph{Squared-norm term.} Decompose $d_k=d_k^{\mathrm{nf}}-d_k^{\mathrm{n}}$ into noiseless and
noise parts. The cross term vanishes because $\Enoise[\varepsilon_i]=0$. Writing
$v_i=w_i(r_k)_i a_i/\norm{a_i}^2$ and using that the $\tau$ indices are i.i.d.\ under $\Eidx$,
\begin{align*}
  \Eidx\,[\norm{d_k^{\mathrm{nf}}}^2]
   =\frac1\tau\Eidx\norm{v_i}^2+\frac{\tau-1}{\tau}\norm{\Eidx v_i}^2
   =\Big\langle r_k,\Big(\tfrac1\tau \mathbf{PW}^2\mathbf{D}^{-2}+\tfrac{\tau-1}{\tau}\mathbf{PWD}^{-2}\mathbf{AA}^\top \mathbf{PWD}^{-2}\Big)r_k\Big\rangle .
\end{align*}
The matrix in brackets equals $\tfrac{2\alpha}{\norm{\mathbf{A}}_F^2}\mathbf{T}$, so
$\Eidx\,[\norm{d_k^{\mathrm{nf}}}^2]\le \tfrac{2\alpha}{\norm{\mathbf{A}}_F^2}\smax(\mathbf{T})\norm{r_k}^2$. For the
noise part, $\Enoise$ acts first ($\Enoise[\varepsilon_i^2]=\sigma_i^2$ on the diagonal, the
cross terms vanishing), then $\Eidx$:
\begin{align*}
  \Eidx\Enoise\,[\norm{d_k^{\mathrm{n}}}^2]
   =\frac1\tau\,\Eidx\Big[\frac{w_i^2\sigma_i^2}{\norm{a_i}^2}\Big]
   =\frac1\tau\sum_{i=1}^m\frac{p_iw_i^2\sigma_i^2}{\norm{a_i}^2}
   =\frac1\tau\Tr(\mathbf{PW}^2\mathbf{\Sigma D}^{-2}).
\end{align*}
Inserting both into Lemma~\ref{lem:onestep} yields the conditional bound
\begin{align*}
  \EE\big[\Df{x_{k+1}^*}{x_{k+1}}{\hatx}\mid x_k\big]\le
   \big(1-\gamma\alpha\eta_k(1-\eta_k\smax(\mathbf{T}))\big)\Df{x_k^*}{x_k}{\hatx}
   +\frac{\eta_k^2}{2\tau}\Tr(\mathbf{PW}^2\mathbf{\Sigma D}^{-2}),
\end{align*}
after applying Assumption~\ref{ass:errorbound}; taking the total expectation $\Etot$ of both
sides via the tower rule gives the stated result. Full details are in
Appendix~\ref{app:lemma-details}.
\end{proof}

\subsection{Spectral bounds and proof of Proposition~\ref{prop:taumono}}\label{sec:spectral}
\begin{lemma}\label{lem:Tbound}
With $\mathbf{T}$ as in~\eqref{eq:Tdef},
\[
  \smax(\mathbf{T})\le\frac{1}{2\tau}\,\smax(\mathbf{W})
     +\frac{\alpha}{2\norm{\mathbf{A}}_F^2}\Big(1-\frac1\tau\Big)\smax^2(\mathbf{A}).
\]
If $\mathbf{W}=\alpha \mathbf{I}$ this holds with equality:
$\smax(\mathbf{T})=\tfrac{\alpha}{2\tau}\big(1+(\tau-1)\smax^2(\mathbf{A})/\norm{\mathbf{A}}_F^2\big)$. The detailed proof of this lemma can be found in~\cite[Lemma 4.10]{tondji2023faster}.
\end{lemma}

\noindent\emph{Proof of Proposition~\ref{prop:taumono}.} For $\mathbf{W}=\alpha \mathbf{I}$, differentiating the
closed form of Lemma~\ref{lem:Tbound},\newline
$\frac{d}{d\tau}\smax(\mathbf{T})=\frac{\alpha}{2\tau^2}\big(\smax^2(\mathbf{A})/\norm{\mathbf{A}}_F^2-1\big)\le 0$ since
$\smax^2(\mathbf{A})\le\norm{\mathbf{A}}_F^2$. Sending $\tau\to\infty$ gives the limit. The contraction factor at
$\eta^\star=1/(2\smax(\mathbf{T}))$ is $1-\gamma\alpha/(4\smax(\mathbf{T}))$, monotone in $\smax(\mathbf{T})$; the ratio of
 $\tau=1$ to $\tau\to\infty$ values of $\smax(\mathbf{T})$ is $\norm{\mathbf{A}}_F^2/\smax^2(\mathbf{A})$.

\hfill$\square$

\subsection{Proof of Theorem~\ref{thm:main}}\label{sec:thmproof}
We sketch the argument; the full details are in Appendix~\ref{app:thm-proof}. To obtain a
contraction rate $q\in(0,1)$ we require $1-\eta_k\smax(\mathbf{T})>0$, i.e.\ $\eta_k<1/\smax(\mathbf{T})$. Under this
condition the one-step bound of Lemma~\ref{lem:descent} can be written as
\begin{equation}\label{eq:Dk-onestep}
  \Etot\,[\Df{x_{k+1}^*}{x_{k+1}}{\hatx}]\le D_k(\eta)\cdot\frac{1}{\tau}\Tr(\mathbf{PW}^2\mathbf{\Sigma D}^{-2}),
\end{equation}
where the error sequence obeys the recurrence
$D_k(\eta)=\big(1-\alpha\gamma\eta_k(1-\eta_k\smax(\mathbf{T}))\big)D_{k-1}(\eta)+\tfrac{\eta_k^2}{2}$.
Minimising $D_k(\eta)$ over the single free variable $\eta_k$ (the second derivative is positive,
so this is a genuine minimum) gives the optimal step
$\eta_k=\alpha\gamma D_{k-1}(\eta)/(1+2\alpha\gamma\smax(\mathbf{T})D_{k-1}(\eta))$. Estimating the unknown
$D_{k-1}(\eta)$ by the deterministic sequence $\beta_k$ with
$ \beta_0=\tau\,\Df{x_0^*}{x_0}{\hatx}/ \Tr(\mathbf{PW}^2\mathbf{\Sigma D}^{-2})$ recovers the stated
form~\eqref{eq:etaopt}; the admissibility $\eta_k<1/\smax(\mathbf{T})$ then holds automatically.
Substituting~\eqref{eq:etaopt} back into the recurrence telescopes it to
$\beta_{k+1}=\beta_k(1-\tfrac{\alpha\gamma\eta_k}{2})$, so $\beta_k$ decreases to
zero. Finally $1$-strong convexity gives $\norm{x_k-\hatx}_2^2\le 2\,\Df{x_k^*}{x_k}{\hatx}$,
which combined with~\eqref{eq:Dk-onestep} and $\beta_k=D_{k-1}(\eta)$ yields the
rate~\eqref{eq:main-rate}. \hfill$\square$

\subsection{Optimal weights via Cauchy--Schwarz}\label{sec:optproof}
We minimise $\Cnoise=\Tr(\mathbf{PW}^2\mathbf{\Sigma D}^{-2})=\sum_i p_iw_i^2\sigma_i^2/\norm{a_i}^2$ over all
$(p,w)$ satisfying the coupling~\eqref{eq:coupling} and $\sum_i p_i=1$, $p_i\ge0$. Eliminating
$w_i=\alpha\norm{a_i}^2/(p_i\norm{\mathbf{A}}_F^2)$,
\[
  \Cnoise=\frac{\alpha^2}{\norm{\mathbf{A}}_F^4}\sum_{i=1}^m\frac{\sigma_i^2\norm{a_i}^2}{p_i}.
\]
Minimizing $\sum_i c_i/p_i$ subject to $\sum_i p_i=1$ is a convex problem solved by
$p_i\propto\sqrt{c_i}$ with optimal value $(\sum_i\sqrt{c_i})^2$ (Appendix~\ref{app:lag}). With
$c_i=\sigma_i^2\norm{a_i}^2$ this gives $p_i^{\mathrm{opt}}\propto\sigma_i\norm{a_i}$ and
\[
  \Cnoise^{\mathrm{opt}}=\frac{\alpha^2}{\norm{\mathbf{A}}_F^4}\Big(\sum_i\sigma_i\norm{a_i}\Big)^2,
  \qquad
  \Cnoise^{\mathrm{uni}}=\frac{\alpha^2}{\norm{\mathbf{A}}_F^2}\sum_i\sigma_i^2 .
\]
With $u=(\sigma_i)_i$ and $v=(\norm{a_i})_i$, Cauchy--Schwarz gives
$(\sum_i\sigma_i\norm{a_i})^2\le(\sum_i\sigma_i^2)(\sum_i\norm{a_i}^2)=\norm{\mathbf{A}}_F^2\sum_i\sigma_i^2$,
hence $\Cnoise^{\mathrm{opt}}\le\Cnoise^{\mathrm{uni}}$, with equality iff the noise standard deviations are strictly proportional to the row norms, i.e., $\sigma_i \propto \|a_i\|$ for all $i=1,\dots,m$. In the special case of homoscedastic noise ($\sigma_i = \sigma$ for all $i$), equality holds if and only if the row norms are identical ($\|a_i\| = \|a_j\|$ for all $i,j$).
This proves Corollary~\ref{cor:optimal}. \hfill$\square$ 

\subsection{Matrix-vector form and unbiasedness}\label{sec:matvec}
The batch $\Bcal_k$ is sampled i.i.d.\ and is in general a \emph{multiset}. Let
$n_i=\sum_{j=1}^\tau\mathbf 1\{I_j=i\}$ be the multiplicity of row $i$, where $I_1,\dots,I_\tau$
are the drawn indices. A row appearing $n_i$ times contributes $n_i$ terms to~\eqref{eq:dk}, each
with an independent fresh sample $\tilde b_i^{(k,1)},\dots,\tilde b_i^{(k,n_i)}$. Writing
$\bar{\tilde b}_i^{(k)}=n_i^{-1}\sum_{j=1}^{n_i}\tilde b_i^{(k,j)} = b_i + \bar \varepsilon^{(k)}_i, \quad \bar \varepsilon^{(k)}_i = n_i^{-1}\sum_{j=1}^{n_i} \varepsilon_i^{(k,j)}$, then
\begin{align*}
  d_k=\frac1\tau\sum_{i=1}^m n_i\frac{w_i}{\norm{a_i}^2}
        \big(\inprod{a_i}{x_k}-\bar{\tilde b}_i^{(k)}\big)a_i
     =\mathbf{A}^\top \mathbf{G_k}\big(\mathbf{A}x_k-\bar{\tilde b}^{(k)}\big),
\end{align*}
where
\begin{align*}
\mathbf{G_k}=\frac1\tau \mathbf{N_k W D}^{-2},\ \mathbf{N_k}=\Diag(n_1,\dots,n_m).
\end{align*}
Taking the batch expectation uses only $\Eidx[n_i]=\tau p_i$, which follows by linearity from
$\Eidx[\mathbf 1\{I_j=i\}]=p_i$, and gives
$\Eidx[\mathbf{G_k}]=\sum_i (p_iw_i/\norm{a_i}^2)e_ie_i^\top=\mathbf{PWD}^{-2}$.
The coupling then collapses $\Eidx[\mathbf{G_k}]$ to $(\alpha/\norm{\mathbf{A}}_F^2)\mathbf{I}$. The averaged noise
$\bar\varepsilon_i^{(k)}$ has variance $\sigma_i^2/n_i$, and multiplying by $n_i$ in the update
reproduces exactly the prefactor $\tau^{-1}\Tr(\mathbf{PW}^2\mathbf{\Sigma D}^{-2})$ of Lemma~\ref{lem:descent}; the
analysis is thus robust to sampling with or without replacement, provided each query returns fresh
noise. In the large-$\tau$ regime one may replace $\mathbf{G_k}$ by $\Eidx[\mathbf{G_k}]=\mathbf{PWD}^{-2}$ for the matrix
part, \emph{provided} the noise is still injected through the realised batch so that the $1/\tau$
suppression is retained.

\section{Numerical experiments}\label{sec:experiments}
We present several experiments to demonstrate the effectiveness of algorithm~\ref{alg:aabk} under various conditions. 

\subsection{ Synthetic experiments}
\label{sec:synthetic-experiments}
In this part, we present several experiments to illustrate the result of Theorem~\ref{thm:main}. For all experiments we consider $f(x) = \lambda \cdot \|x\|_1 + \frac{1}{2}\|x\|^2_2$, where $\lambda$ is the sparsity parameter which gives us $\nabla f^*(x) = S_{\lambda}(x)$. Note that $f$ is $1$-strongly convex but not smooth. Synthetic data for the experiments are generated as follows:
all elements of the data matrix $\mathbf{A} \in \RR^{m\times n}$ ($m=2000, n=100$) are chosen independently and identically distributed from the standard normal distribution $\mathcal{N}(0, 1)$. To construct the sparse solution $\hat x \in \RR^n$, we generate a random vector with $s \;(\geq 0)$ non-zero entries from the standard normal distribution and we set it as $\hat x$ and the corresponding right-hand side is $b = \mathbf{A}\hat x \in \RR^m$. 

The benefit of the ABK method over the standard RK has already been shown in~\cite{tondji2024adaptive}, so in this part the comparison will focus more on the methods using adaptive step-size. To measure the performance, we used the plot of the relative error $\|x_k -\hat{x}\|_2 / \| \hat{x}\|_2$ and the average relative residual error $\| \mathbf{A}x_k - b\|_2/\|b\|_2$ in a semilog plot against the number of iterations.\\

To unify the implementation and for fair-comparison, we implement all the methods below under the same algorithm. Concretely, at each iteration $k$ we sample a batch of block indices $\Bcal_k\subseteq\{1,\dots,M\}$ of
size $\tau$, i.i.d.\ from $p$ (with replacement), with $1 \leq M \leq m $ being the number of blocks and form the weighted average of the individual
Kaczmarz directions,
\begin{equation}
\label{eq:dk1}
      d_k = \frac{1}{\tau} \sum_{i \in \mathcal{B}_k} w_i\frac{ A_{(i)}^{\top}(A_{(i)}x_k - \tilde{b}_{(i)})}{\|A_{(i)}\|_2^2}.
\end{equation}
In this way we keep Algorithm~\ref{alg:aabk} as it is but replace the update of the direction vector $d_k$ with the one in Eq~\eqref{eq:dk1}. We compare the following six methods:
\begin{enumerate}
\item \textbf{ABK}: The Adaptive Bregman-Kaczmarz method~\cite{tondji2024adaptive}, which can be derived from Algorithm~\ref{alg:aabk}(eq~\eqref{eq:dk1}), with number of block $M=100,$ 
adaptive  stepsize $\eta_{k}$ defined by~\eqref{eq:etaopt}, batch size $\tau=1,$ an exact value for $\beta_{0}$ (using the knowledge of the ground truth solution).
\item \textbf{AABK\_uni}: The Adaptive Averaged Bregman–Kaczmarz method, i.e. Algorithm~\ref{alg:aabk}, row-wise and Uniform weights with $M=m$, $\mathbf{W}=\alpha \mathbf{I},$ an adaptive stepsize $\eta_{k}$ defined by~\eqref{eq:etaopt}, batch size $\tau=20,$ an exact value for $\beta_{0}$ (using the knowledge of the ground truth solution).
\item \textbf{AABK\_gen}: The Adaptive Averaged Bregman–Kaczmarz method, i.e. Algorithm~\ref{alg:aabk}, row-wise and Optimal noise-aware weights with $M=m,$  $w_i\propto\norm{a_i}/\sigma_i$,
$p_i\propto \sigma_i\norm{a_i}$, an adaptive  stepsize $\eta_{k}$ defined by~\eqref{eq:etaopt}, batch size $\tau=20,$ an exact value for $\beta_{0}$ (using the knowledge of the ground truth solution).
\item \textbf{RK}: The Randomized Kaczmarz method, block-wise, i.e. Algorithm~\ref{alg:aabk} with $M=100,$ batch size $\tau=1,$ constant step size $\eta_{k}=1$ (hence, neither $\gamma$ nor $\beta_0$ is needed, especially no ground truth knowledge is needed).
\item \textbf{RBKA\_uni}: The Randomized Batch Kaczmarz method, row-wise with uniform weights, i.e. Algorithm~\ref{alg:aabk} with $M=2000,$ batch size $\tau=20,$ $\eta_{k}=1$ (hence, neither $\gamma$ nor $\beta_0$ is needed, especially no ground truth knowledge is needed).
\item \textbf{RBKA\_gen}: The Randomized Batch Kaczmarz method, row-wise with optimal noise-aware weights, i.e. Algorithm~\ref{alg:aabk} with $M=2000,$ $w_i\propto\norm{a_i}/\sigma_i$,
$p_i\propto \sigma_i\norm{a_i}$, batch size $\tau=20,$  $\eta_{k}=1$ (hence, neither $\gamma$ nor $\beta_0$ is needed, especially no ground truth knowledge is needed).
\end{enumerate}

For fair comparison, the batch size $\tau$ for RBKA/AABK and the block-size (bs) for RK/ABK are set equal so that all algorithms use the same number of rows in each iteration. The right-hand side is corrupted non-uniformly: a randomly chosen subset of $ N_{\mathrm{large}} =100$ equations $(5\%)$ is assigned a large noise level $\sigma_i = \text{scale}\cdot\sigma_0 = 0.5$ (with $\text{scale}=10$ and $\sigma_0=0.05$), while the remaining $1900$ equations $(95\%)$ retain $\sigma_i = \sigma_0 = 0.05$, producing a sparse but strongly heterogeneous noise profile. To tune the hyperparameters, we first fixed $\lambda=0.5$ for both ABK and AABK, we independently select the best $\gamma$ for each method from the grid $\{0.001, 0.005, 0.01, 0.05, 0.1, 1, 2\}$. As shown in Figure~\ref{fig:combined_comparison} and~\ref{fig:combined_comparison1} , RK and both RBKA variants stagnate at the noise floor: their constant step size $\eta_k = 1$ prevents the iterates from progressing beyond the level imposed by the measurement noise. ABK, equipped with an adaptive step size, converges below this floor. AABK with uniform weights further improves over ABK by leveraging the batch structure to reduce variance. AABK with general (noise-adaptive) weights achieves the lowest error in both the $\tau=1$ and $\tau=20$ settings, and the gap over all baselines widens significantly as the batch size increases.

\begin{figure}[htb!]%
    \centering
    \includegraphics[width=0.47\textwidth]{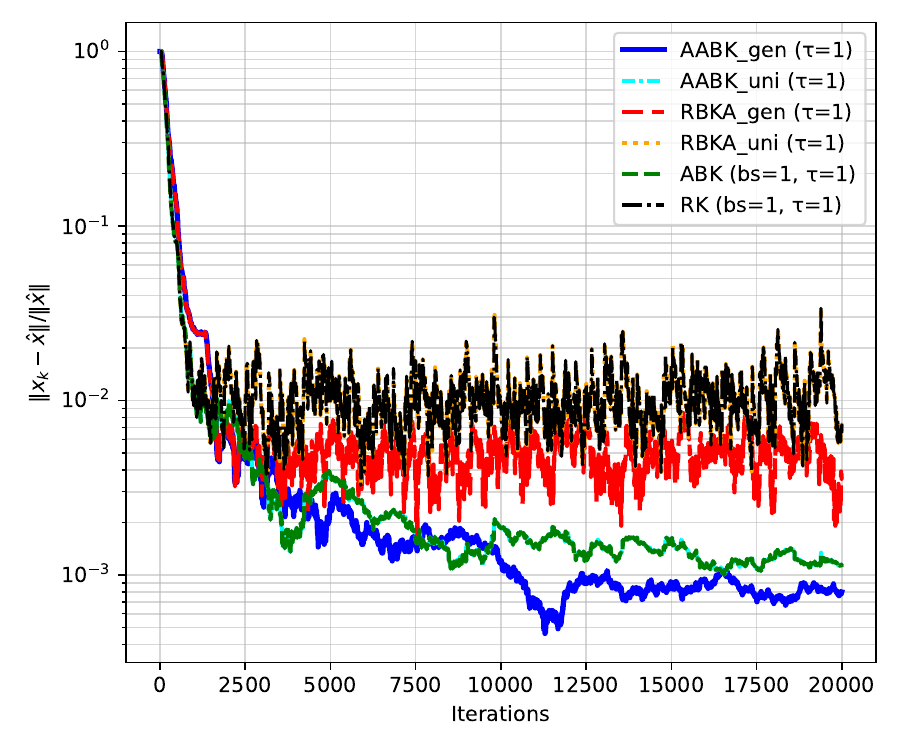}%
    \hfill%
    \includegraphics[width=0.47\textwidth]{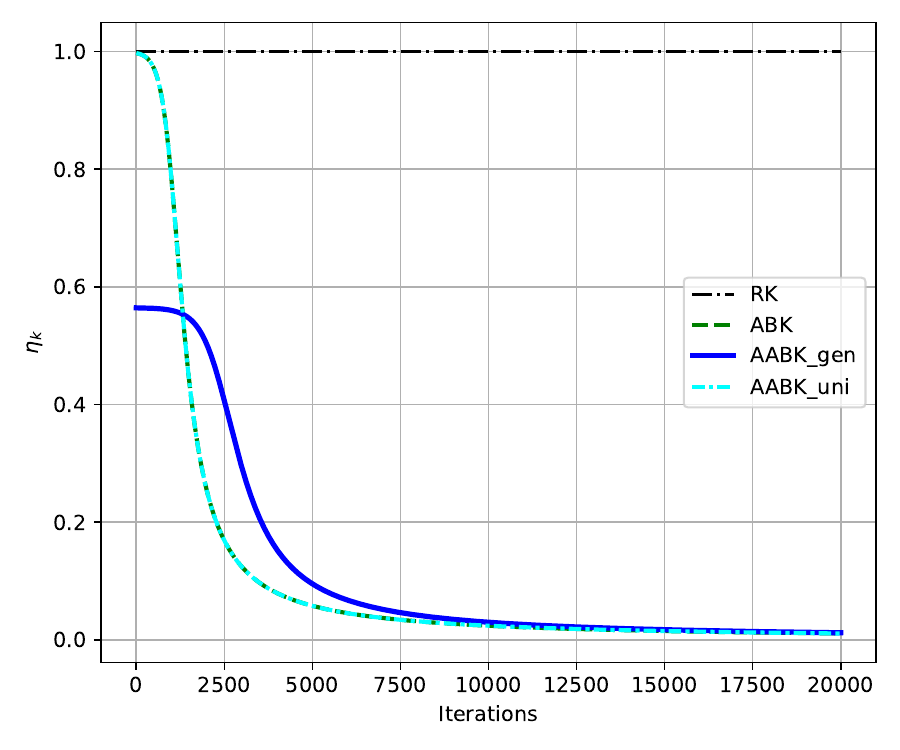}
    
    \vspace{1ex} 
    
    \includegraphics[width=0.47\textwidth]{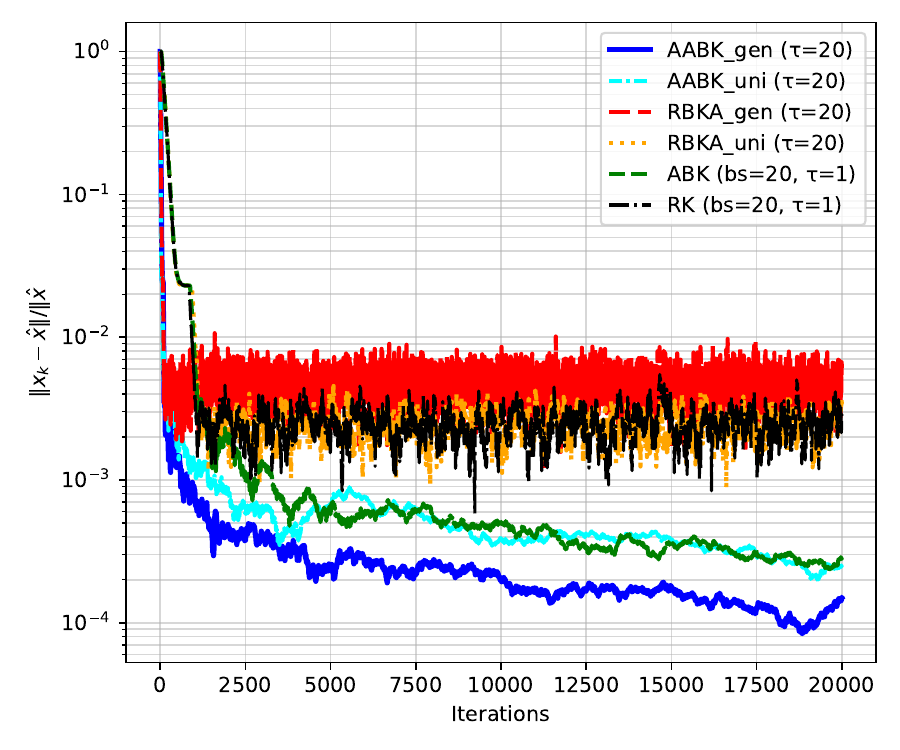}%
    \hfill%
    \includegraphics[width=0.47\textwidth]{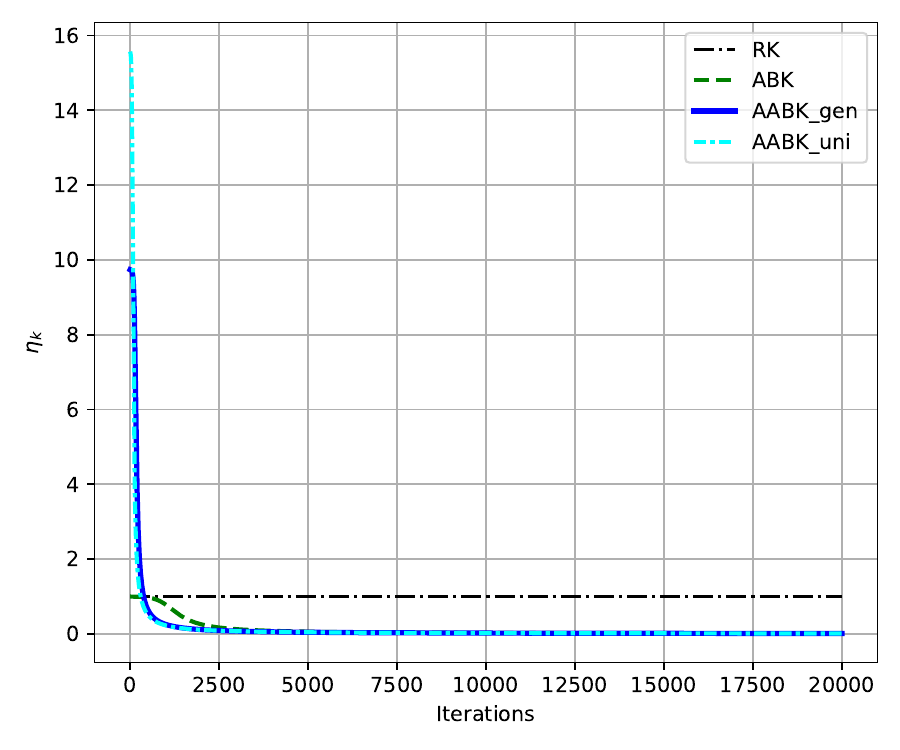}%
    
    \caption{A comparison of the six methods as described in Section~\ref{sec:synthetic-experiments} for  $m=2000$, $n=100$, sparsity $s=10$, $\alpha=1$, $\gamma=0.01$, $\lambda=0.5$. 
    Left column: relative error $\|x_k-\hat{x}\|/\|\hat{x}\|$. 
    Right column: adaptive step size $\eta_k$. 
    Within each subfigure, the top plots correspond to the row-wise setting ($\tau=1$, $\text{bs}=1$), while the bottom plots correspond to the batched setting ($\tau=20$, $\text{bs}=20$ for RK/ABK). 
    }%
    \label{fig:combined_comparison}%
\end{figure}

\begin{figure}[htb!]
        \centering
        \includegraphics[width=0.9\linewidth]{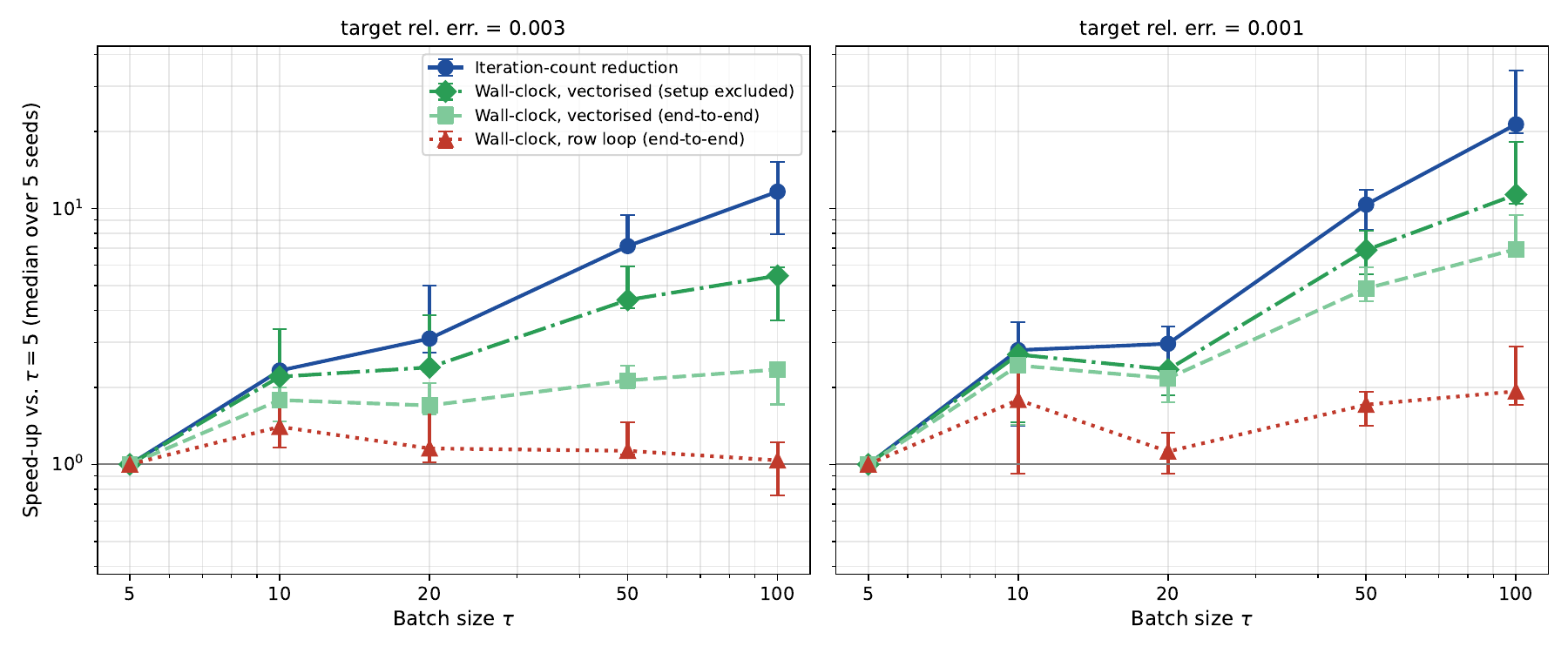}
    \caption{Wall-clock runtime and parallel speedup across batch sizes $\tau$ for experiment in Section~\ref{sec:synthetic-experiments}.}
    \label{fig:plot}
\end{figure}

\begin{figure}[htb!]%
    \centering
    \includegraphics[width=0.47 \textwidth]{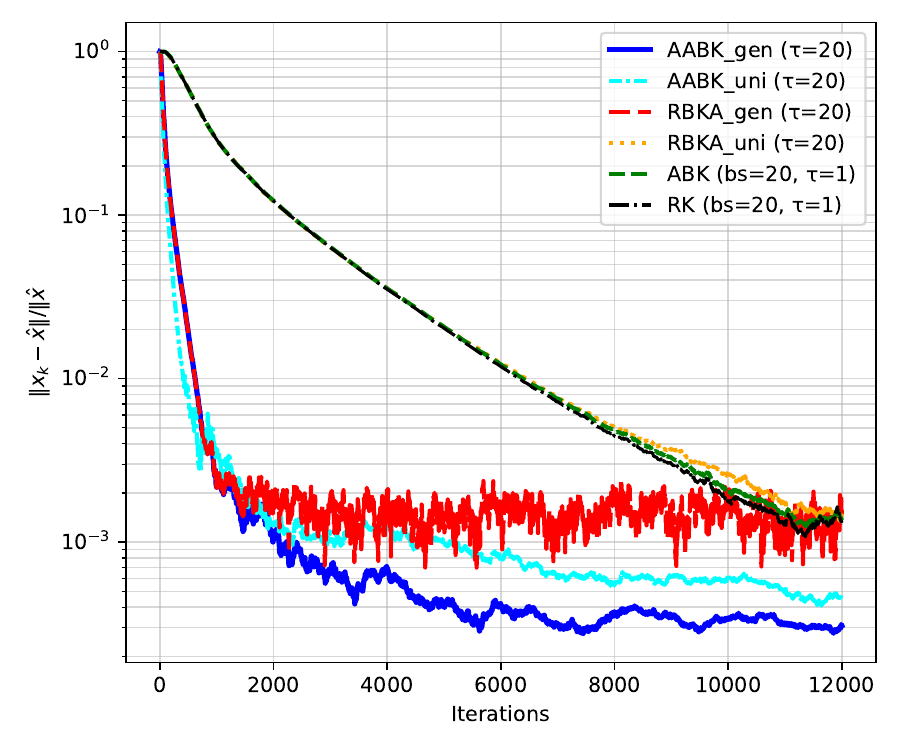}
    \qquad
    \includegraphics[width=0.47 \textwidth]{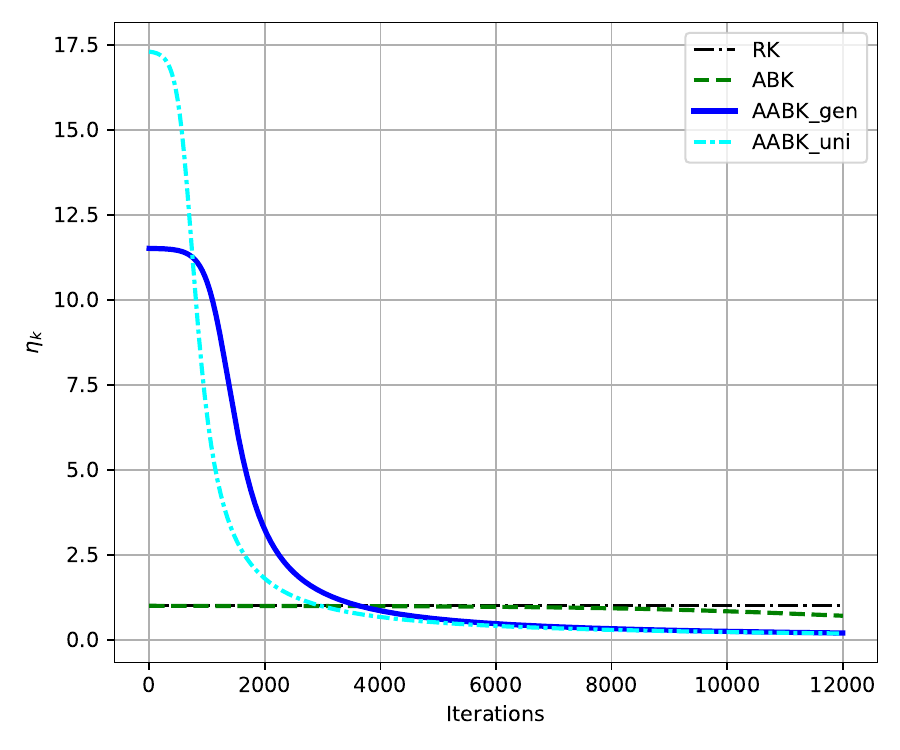}\vspace{-2ex}
    \caption{A comparison of the six methods as described in Section~\ref{sec:synthetic-experiments} for  $m=400$, $n=600$, sparsity $s=10$, $\alpha=1$, $\gamma=0.001$, $\lambda=0.5$. 
    Left column: relative error $\|x_k-\hat{x}\|/\|\hat{x}\|$. 
    Right column: adaptive step size $\eta_k$. 
    }%
    \label{fig:combined_comparison1}%
\end{figure}

\subsubsection{Effect of the number of corrupted rows and the per-iteration row budget}
\paragraph{Setup:} We consider two experiments that assess how ABK (uniform weights, $\tau=1$) and AABK\_gen (optimal noise-aware weights) respond to the noise structure, keeping the per-iteration row budget equal for both methods throughout.
\begin{enumerate}
    \item In the first experiment (left), we fix the system matrix $\mathbf{A}$ and vary the number of corrupted rows $N_{\mathrm{large}} \in \{5, 50, 200, 500, 1000\}$, each receiving noise scaled by a factor of $10$ (hollow markers) or $800$ (filled markers), and record the final relative error after a fixed iteration count.

    \item In the second experiment (right), we fix $N_{\mathrm{large}}=100$ corrupted rows and vary the per-iteration row budget $k \in \{5, 50, 100, 500, 1000\}$ under two noise scales, $10$ (hollow markers) and $500$ (filled markers). ABK uses $\mathrm{block\_size}=k$, $\tau=1$; AABK-gen uses $\tau=k$, $\mathrm{block\_size}=1$, so both access exactly $k$ rows per iteration.
\end{enumerate}

\begin{figure}[htb]%
    \centering
    \includegraphics[width=0.47 \textwidth]{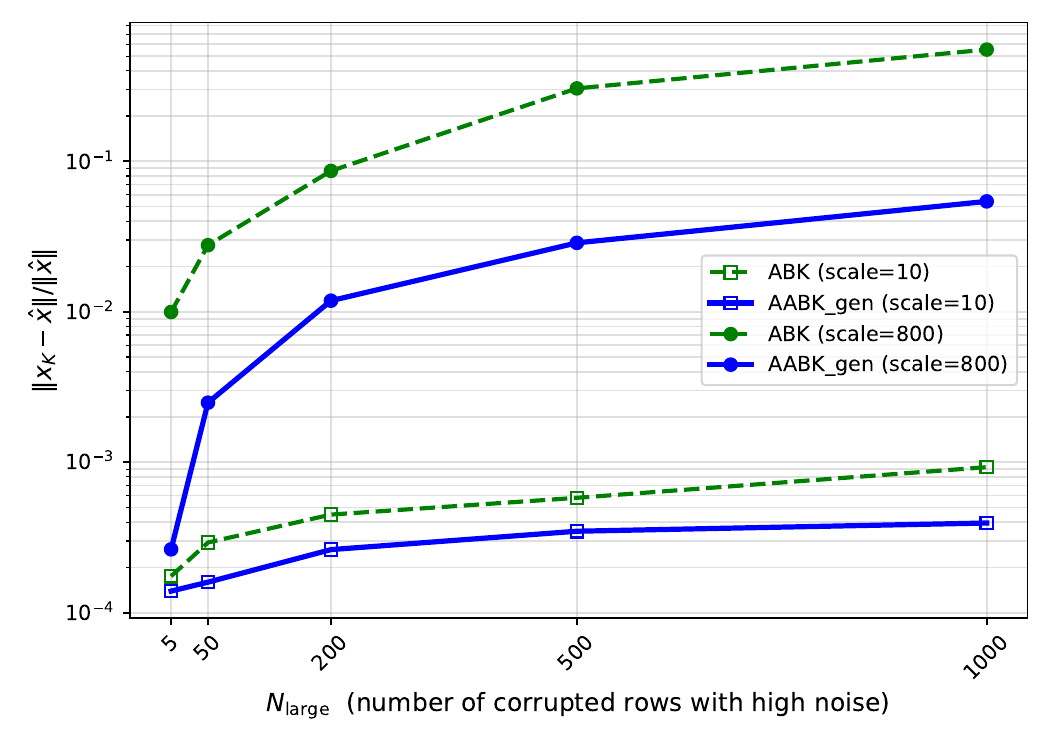}
    \qquad
    \includegraphics[width=0.47 \textwidth]{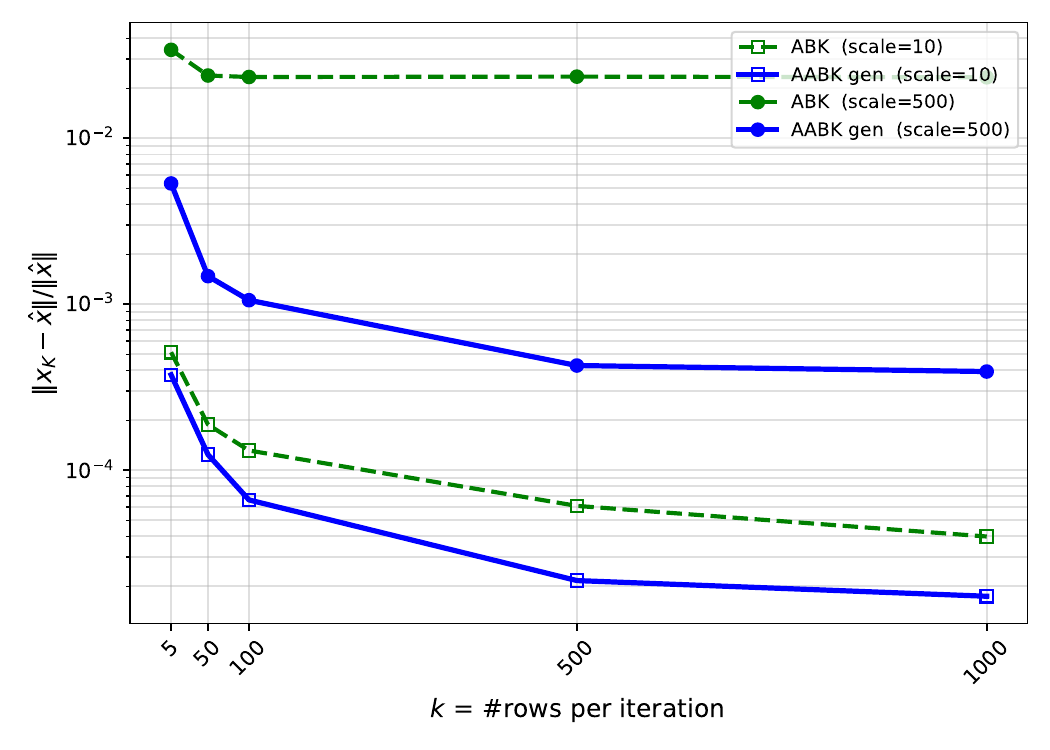}\vspace{-2ex}
    \caption{A comparison of the ABK and AABK methods as described in Section~\ref{sec:synthetic-experiments} for  $m=2000$, $n=100$, sparsity $s=10$, $\alpha=1$, $\gamma=0.01$, $\lambda=0.5$. Left: Effect of the number of corrupted rows $N_{\mathrm{large}}$. Right: Effect of the per-iteration row budget. 
    }%
    \label{fig:example1}%
\end{figure}

\paragraph{Interpretation}: Figure~\ref{fig:example1} reports both experiments: the left panel shows the final relative error as a function of the number of corrupted rows $N_{\mathrm{large}}$ with large noise, and the right panel as a function of the per-iteration row budget $k$. Both reveal the same structural limitation of uniform weights. In the left panel, ABK's error grows almost linearly with $N_{\mathrm{large}}$ and deteriorates sharply at scale $800$, since uniform weights cannot discount the heavily corrupted rows; AABK\_gen's noise-adapted weights (high $\sigma_i \Rightarrow$ small $w_i$) keep the error an order of magnitude lower across all regimes. In the right panel, at moderate noise (scale = 10) both methods improve as $k$ grows, but AABK\_gen gains considerably more. At high noise (scale $=500$), ABK's error is essentially flat: grouping more rows into a single block provides no benefit when all rows are weighted equally. AABK\_gen, by contrast, drops by nearly two orders of magnitude as $k$ increases; a larger $\tau$ averages out more noise per step, directly reducing the bias term in the convergence bound. Together, the two plots show that the gain from a larger row budget is only realised when weights are adapted to the noise structure.

\begin{remark}[Iterations vs.\ computation time]\label{rem:runtime}
As illustrated in Figure~\ref{fig:plot}, increasing the batch size $\tau$ reduces the number of iterations needed to reach a given accuracy, roughly in proportion to $\tau$. Since the total arithmetic work $\tau k$ remains of the same order, the reduction in iterations becomes an actual reduction in computation time only when the $\tau$ independent row updates in each iteration are evaluated in parallel. The same behaviour was observed and discussed for the averaged block sparse Kaczmarz method in~\cite{tondji2023faster}.
\end{remark}

\subsection{CT reconstruction under heterogeneous noise}\label{sec:ct}

\emph{Setup.}
As an example on computerized tomography (CT), we used the implementation of the Radon transform from the \texttt{Python} library \texttt{skimage} and use it to build a system matrix for a parallel beam CT for a phantom of size $N\times N$ with $N=50$ and with $60$ equispaced angles. Hence, the system matrix $\mathbf{A}$ has size $3000\times 2500$, i.e. $m=3000$ and $n=2500$. A random subset of $30$ rows corresponding to $1\%$ of the total numbers of rows $m$ are corrupted by very high heterogeneous Gaussian noise of  $\sigma_i=10\,\sigma_{\mathrm{base}} \approx 33.779$ (i.e., scale$=10$),  while the rest is only corrupted by a noise level of $1\%$ i.e we choose $\sigma_{\mathrm{base}}=10^{-2}\norm{b}_2 \approx 3.3779$. Fresh noise $\varepsilon_i^{(k)}\sim\mathcal N(0,\sigma_i^2)$ is
drawn for each iteration. 
For reproducibility we fixed the random seed $42$ and a relaxation parameter $\alpha=1$, we used different methods for reconstruction, each run for 40 epochs, i.e. for $120000$ iterations and using the function $f(x)=\lambda\norm{x}_1+\tfrac12\norm{x}_2^2$ with $\lambda=30$ since the phantom is $63.5\%$ sparse. For the comparison over the reconstruction, we use the following metrics: relative error, residual errors, $\mbox{SSIM}$ and $\mbox{PSNR}.$ As in synthetic experiments, we set the batch size $\tau$ for RBKA/AABK and the block-size for RK/ABK vice versa equal (see Table \ref{tab:methods}) so that all algorithms use the same budget in each iteration.

\begin{table}[htb!]\centering
\caption{Algorithms compared in the CT experiment.}\label{tab:methods}
\resizebox{\textwidth}{!}{%
\begin{tabular}{lccccccc}
\toprule
Method & Sampling & Batch size ($\tau$) & Block size  & $\gamma$ & $\beta_0$ & Step size ($\eta_k$)\\
\midrule
RK   & uniform     & $1$ & $50$ & / & / &  $\eta_k=1$\\
ABK  & uniform     & $1$ & $50$  & $10^{-4}$ & $1.08 \times 10^{-4}$ & $\eta_k$ from~\eqref{eq:etaopt}\\
RBKA\_uni       & uniform & $50$ & $1$ & / & / &  $\eta_k=1$\\
RBKA\_gen      & noise-aware & $50$ & $1$ & / & / &  $\eta_k=1$\\
AABK\_uni      & uniform & $50$ & $1$ & $10^{-4}$  & $5.43 \times 10^5$   & $\eta_k$ from~\eqref{eq:etaopt}\\
AABK\_uni\_heur & uniform & $50$ & $1$ & $1.1 \times 10^{-5}$  &  $7.70 \times 10^7$    & $\eta_k$ from~\eqref{eq:etaopt}\\
AABK\_gen & noise-aware & $50$ & $1$ & $10^{-4}$  &   $9.20 \times 10^5$    & $\eta_k$ from~\eqref{eq:etaopt}\\
AABK\_gen\_heur  & noise-aware & $50$ & $1$ & $4.1 \times 10^{-5}$  &  $1.80 \times 10^{10}$   & $\eta_k$ from~\eqref{eq:etaopt}\\

\bottomrule
\end{tabular}
}
\end{table}

\emph{Methods.} We compare four schemes (Table~\ref{tab:methods}): RK (uniform, fixed step), ABK
(uniform, adaptive step), RBKA (uniform and noise-aware weights, fixed step) and AABK (uniform and noise-aware weights,
adaptive step). The noise-aware methods use $w_i\propto\norm{a_i}/\sigma_i$,
$p_i\propto \sigma_i\norm{a_i}$. The heuristic hyperparameters $\tilde{\gamma}$ and $\tilde{\beta}_0$ for adaptive step size methods are estimated following Section~\ref{sec:hyper} with $N_0=10\, 000$ and the $N_1 = 25\, 000,$ for each methods (block-wise, row-wise with uniform and noise-aware weights), namely AABK\_uni\_heur and AABK\_gen\_heur. Methods (ABK, AABK\_uni, and AABK\_gen) use a guessed $\gamma=0.0001$ with exact value of $\beta_0$,  computed for each method, while the heuristic variants estimate both $\gamma$ and $\beta_0$ as in \S\ref{sec:hyper}. 
\begin{figure}[htb!]
    \centering
    \includegraphics[width=0.6\linewidth]{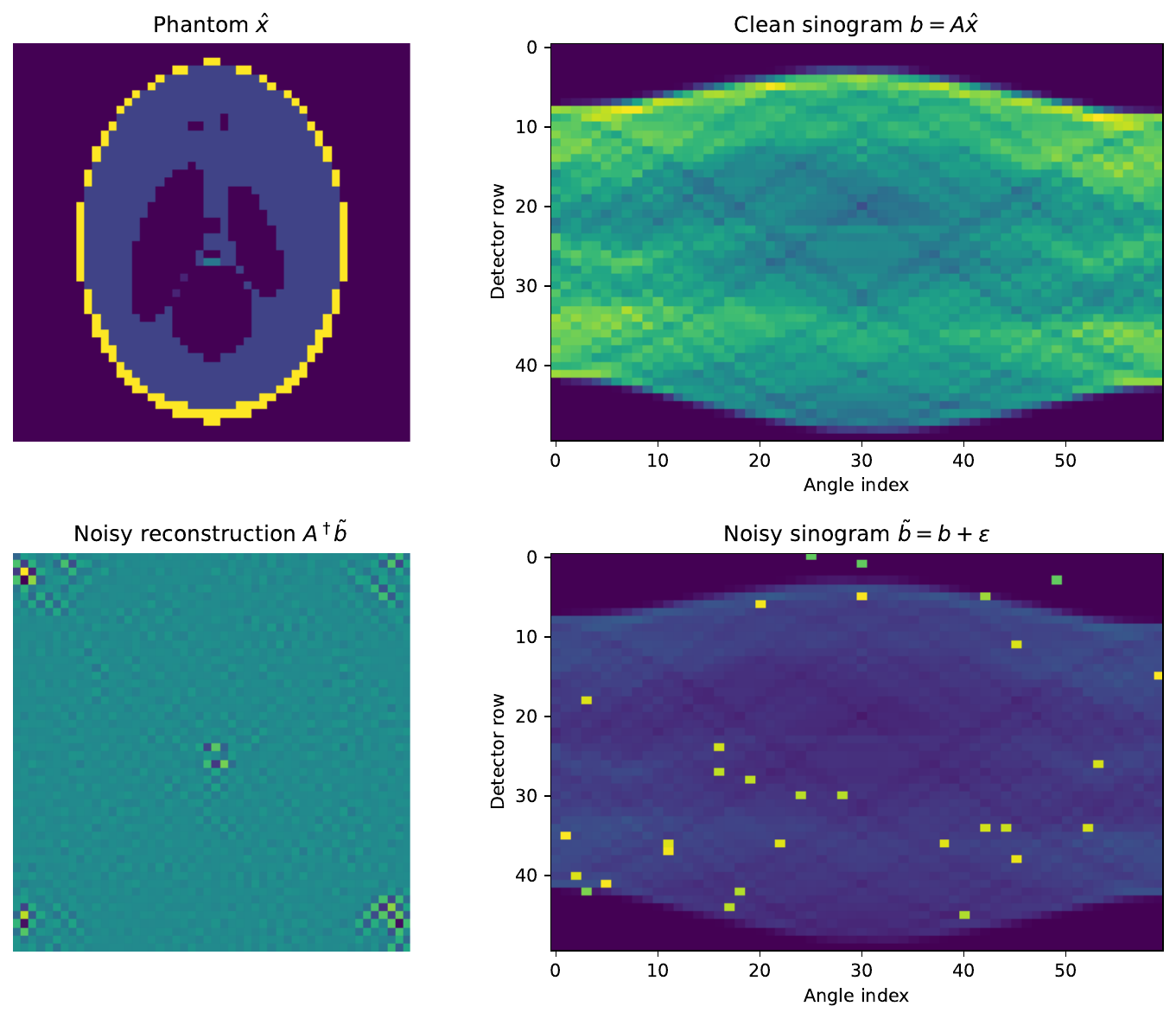}
    \caption{Pseudo-inverse reconstruction from noisy sinogram.}
    \label{fig:CT_phantom_sinogram}
\end{figure}

\begin{figure}[htb!]
    \centering
    \begin{minipage}[b]{0.49\textwidth}
        \centering
        \includegraphics[width=1.0\linewidth]{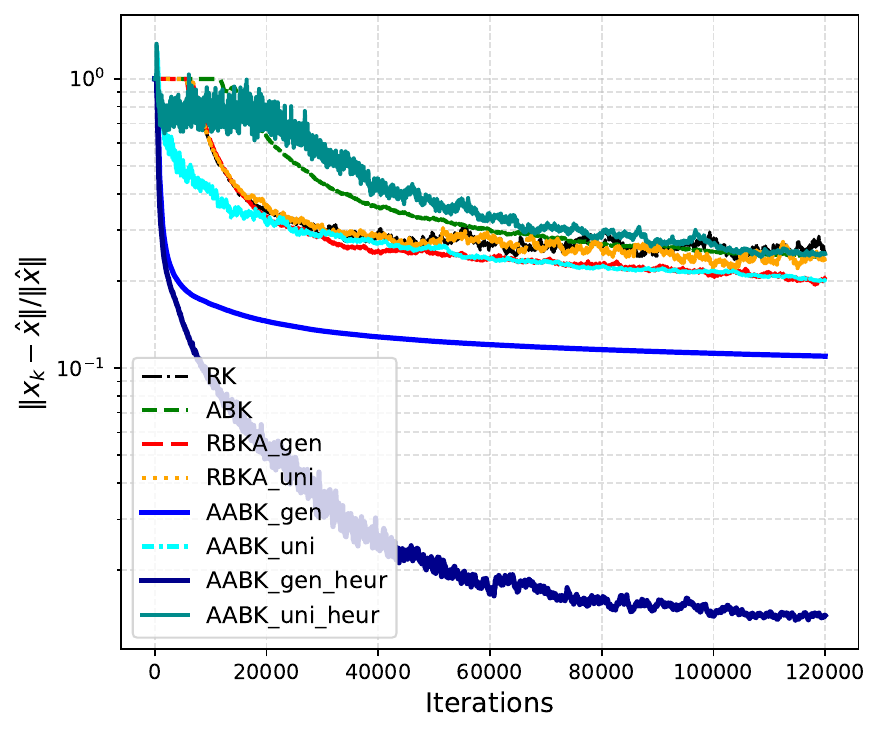}
    \end{minipage}
    \hfill 
    \begin{minipage}[b]{0.49\textwidth}
        \centering
        \includegraphics[width=1.0\linewidth]{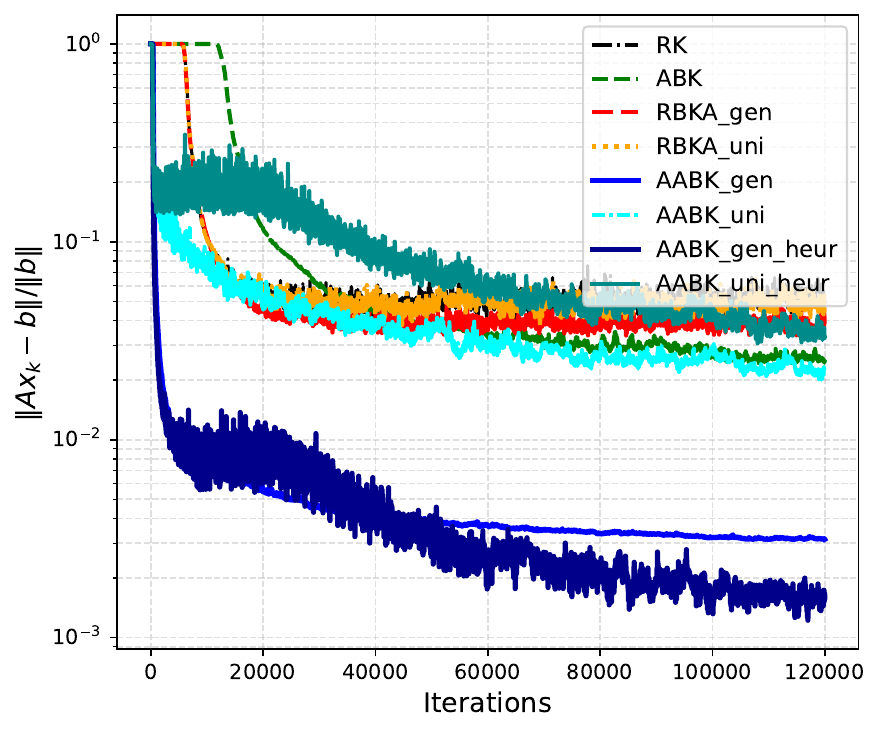}
    \end{minipage}
    \caption{Comparison of relative error (left) and residual error (right) for different methods in the CT experiment.}
    \label{fig:CT_errors_comparison}
\end{figure}
\emph{Results.} 
Figure \ref{fig:CT_phantom_sinogram} illustrates how the phantom projection  via the Radon transform over $60$ angles, yields a clean sinogram $b = \mathbf{A}\hat x$. One realisation of the noisy sinogram $\tilde{b} = b + \epsilon$ is built with heteroscedastic noise described above. The naive pseudo-inverse $\mathbf{A}^\dagger\tilde{b}$ reconstruction  illustrates that direct inversion is unstable under noise background and therefore motivates the regularised iterative methods compared thereafter.

Figure~\ref{fig:CT_errors_comparison} shows that AABK with optimal noise-aware
weights converges fastest and attains the smallest error at every iteration.
The non-adaptive methods exhibit a transient spike early on in both relative and
residual errors, caused by sampling heavily corrupted rows while the residual is
still large; the adaptive step in ABK and AABK suppresses this overshoot by
automatically down-weighting large-noise updates. Among row-wise methods,
optimal noise-aware weights consistently outperform their uniform counterparts
(AABK\_gen vs.\ AABK\_uni and RBKA\_gen vs.\ RBKA\_uni), confirming the
effectiveness of the noise-aware weighting scheme under the heteroscedastic
noise profile of this CT setting. The residual curves mirror the relative-error
trends. The heuristic AABK variants remain competitive with their exact-hyperparameter counterparts throughout the run, confirming that the estimators of \S\ref{sec:hyper} recover usable values of $\gamma$ and $\beta_0$ without any knowledge of the ground-truth solution $\hat x$: the only quantities the
user must supply are the sparsity parameter $\lambda$ and the two window indices
$N_0$ and $N_1$ delimiting the decay and plateau phases used to estimate
$\gamma$ and $\beta_0$. We note that the very large heuristic $\beta_0$ (Table~\ref{tab:methods}) keeps the step in its aggressive constant-step phase for longer, which can even improve the final metrics; the differences between the exact and heuristic rows in Table~\ref{tab:Metrics} therefore reflect the choice of $\gamma$ and $\beta_0$, not the fidelity of the estimation alone. We observed that the primary failure mode for both parameters is forcing a premature collapse into the $\mathcal{O}(1/k)$ decay regime caused by either \emph{underestimating $\beta_0$} or \emph{overestimating $\gamma$} which severely slows convergence before the residual reaches the noise floor. Conversely, overestimating $\beta_0$ is benign and speeds up early progress, whereas underestimating $\gamma$ primarily degrades late-stage noise suppression.

Moreover, the relative and residual error curves of the
adaptive methods continue to decrease through the final epoch and are expected
to improve further with more iterations. This is notable given that AABK
operates only on corrupted measurements each query returning a fresh,
independent perturbation of the true right-hand side yet still drives the
error toward the noise-free solution.
\begin{figure}[htb!]
    \centering
    \includegraphics[width=1.0\linewidth]{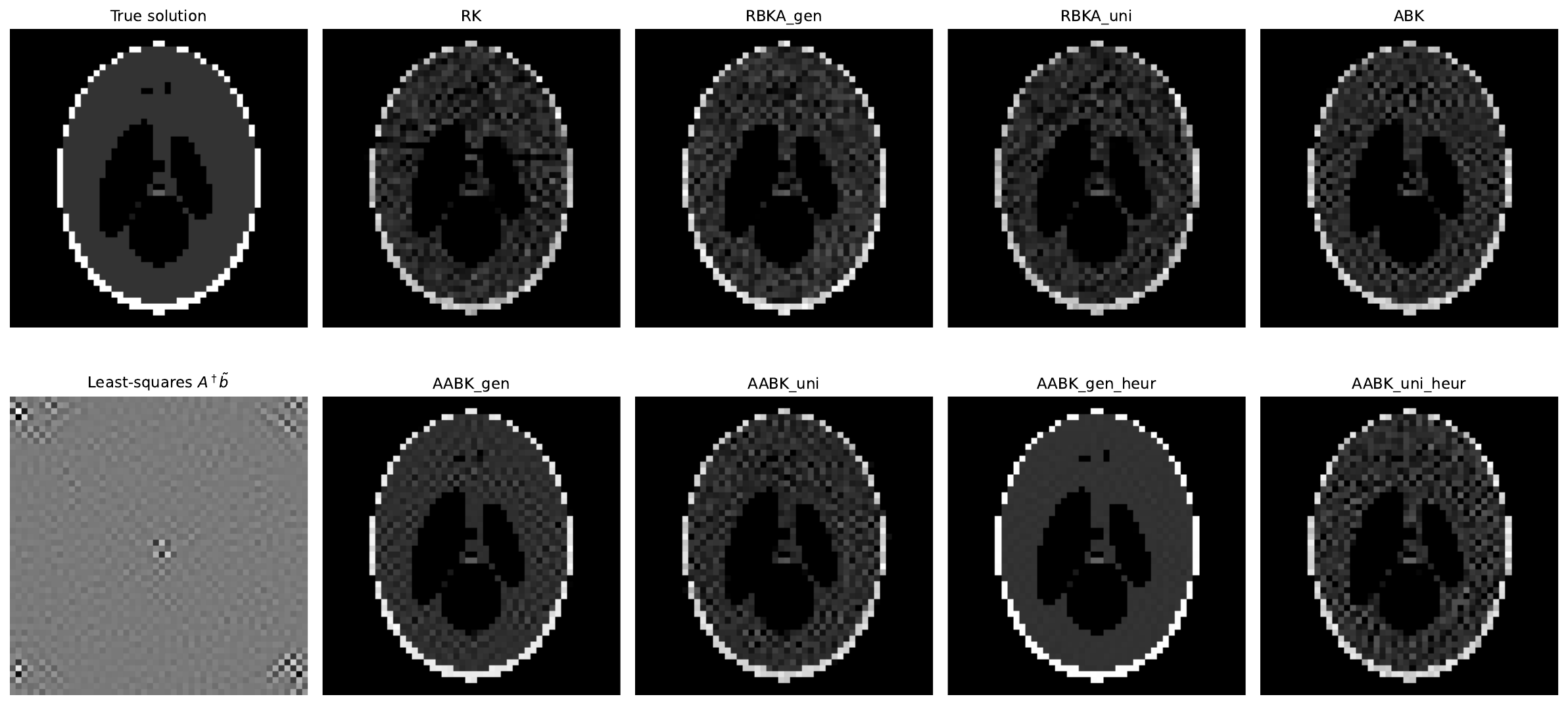}
    \caption{Ground truth solution and reconstruction by the different methods  under heterogeneous noise.}
    \label{fig:CT reconstruction}
\end{figure}

Figure \ref{fig:CT reconstruction} shows the ground truth image alongside the least-squares solution $\mathbf{A}^\dagger\tilde{b}$ and the reconstructions produced by different methods. The least-squares solution is severely corrupted by noise and yields no recognizable anatomical structure. The RK reconstruction exhibits strong graininess and visible horizontal streak artifacts, with the internal structures barely discernible. ABK and the two RBKA variants (uniform and general) partially recover the overall shape but retain significant noise in both the background and the non-zero regions. The AABK methods, by contrast, produce markedly cleaner reconstructions: the background is substantially suppressed and the anatomical structures of the phantom are more faithfully recovered, owing to the combination of the sparsity-promoting regularizer $f$ and the adaptive noise-weighted averaging. 
Both the uniform and general variants of AABK achieve comparable visual quality, as do their heuristic counterparts, confirming that the heuristic parameter selection does not degrade reconstruction performance. 

To quantitatively assess reconstruction quality we report SSIM and PSNR in Table~\ref{tab:Metrics}, two standard image quality metrics that measure perceptual similarity and signal fidelity relative to the ground truth.

\begin{table}[htb!]\centering
\caption{SSIM and PSNR metrics of all methods compared in the CT experiment.}\label{tab:Metrics}
\setlength{\tabcolsep}{4pt}
\resizebox{\textwidth}{!}{%
\begin{tabular}{lcccccccc}
\toprule
Method & RK & ABK & RBKA\_uni & RBKA\_gen & AABK\_uni & AABK\_uni\_heur & AABK\_gen & AABK\_gen\_heur \\
\midrule

SSIM ($\uparrow$) & $0.8563$ & $0.8612$ & $0.8645$ & $0.8856$ & $0.8863$ & $0.8612$ & $\mathbf{0.9409}$ & $\mathbf{0.9975}$ \\\\
PSNR ($\uparrow$) & $24.4112$ & $24.7586$ & $24.8529$ & $26.1943$ & $26.3740$ & $24.5479$ & $\mathbf{31.6334}$ & $\mathbf{49.5461}$ \\
\bottomrule
\end{tabular}%
}
\end{table}

\emph{Discussion.} Two orthogonal gains combine in AABK. \emph{Noise-aware sampling} assigns higher selection probabilities to noisier rows to ensure frequent queries that accelerate empirical zero-mean cancellation through averaging, while assigning them smaller weights to prevent individual noisy updates from destabilizing the iterate trajectory.
The \emph{adaptive step} tunes $\eta_k$ to the current energy $\beta_k$, preventing
overshoot when the residual is large and accelerating near the solution (shared with ABK). Only
AABK combines both, explaining its superior convergence and reconstruction quality.

\section{Conclusion}
\label{sec:conclusion}
We introduced Adaptive Averaged Bregman--Kaczmarz (AABK), which combines the
variance reduction of averaged Kaczmarz updates with the exact recovery of
adaptive stepping under independent noise, in a single method whose convergence
is governed by one positive-semidefinite matrix. This reduction settles a
question left open by the block ABK analysis: with an averaged rather than
summed update, larger batches provably help, and the maximal gain is exactly the
stable rank of the system matrix. On top of this, a noise-aware weighting
minimises the noise constant among all admissible couplings and is strictly
better than uniform weighting whenever the noise is not proportional to the row
norms. Crucially, the noise-optimal weights automatically preserve the batch-size
acceleration, so the two gains compose. Our step-size analysis further shows that
AABK interpolates on its own between a fast initial phase and a vanishing tail
that drives the error exactly to zero, with hyperparameters estimable from a
single auxiliary run. The quality of these estimates still rests on two
user-chosen indices, $N_0$ and $N_1$, which delimit the phases used to estimate
$\gamma$ and $\beta_0$ respectively; although they play a significant role, good
values can be identified by inspecting the residual decay of a
standard Bregman--Kaczmarz run---selecting the steady linear-decay regime for
$\gamma$ and the stagnation regime for $\beta_0$. Experiments on CT
reconstruction under sparse, heterogeneous noise corroborate every prediction.
Several directions remain open. Our guarantees are in expectation;
high-probability bounds via martingale concentration are a natural next step. The
coupling between sampling and weights is imposed for tractability, and the fully
optimal pair may lie outside it.


\bibliographystyle{plain}  
\bibliography{ref}   

\newpage

\appendix

\section*{Appendix}

\section{Monotone acceleration for the optimal weights: proof of Proposition~\ref{prop:taumono}(iii)}
\label{app:optmono}
We show that the optimal noise-aware weights of Corollary~\ref{cor:optimal} satisfy the
general-weight condition of Proposition~\ref{prop:taumono}(ii) automatically, with strict slack.

\paragraph{The optimal weights.}
Under the coupling~\eqref{eq:coupling}, $w_i=\alpha\norm{a_i}^2/(p_i\norm{\mathbf{A}}_F^2)$. Substituting
the optimal sampling distribution $p_i^{\mathrm{opt}}=\sigma_i\norm{a_i}/S$ with
$S:=\sum_{j=1}^m\sigma_j\norm{a_j}$ gives
\[
  w_i=\frac{\alpha\norm{a_i}^2}{\norm{\mathbf{A}}_F^2}\cdot\frac{S}{\sigma_i\norm{a_i}}
     =\frac{\alpha\,S}{\norm{\mathbf{A}}_F^2}\cdot\frac{\norm{a_i}}{\sigma_i},
\]
so $w_i\propto\norm{a_i}/\sigma_i$. As $\mathbf{W}=\Diag(w_1,\dots,w_m)$ is diagonal with nonnegative
entries,
\[
  \smax(\mathbf{W})=\max_{1\le i\le m} w_i=\frac{\alpha\,S}{\norm{\mathbf{A}}_F^2}\,\max_{1\le i\le m}\frac{\norm{a_i}}{\sigma_i}.
\]

\paragraph{A lower bound on $\smax(\mathbf{W})$.}
Let $r\in\arg\max_i \norm{a_i}/\sigma_i$, so that $\norm{a_r}/\sigma_r\ge\norm{a_j}/\sigma_j$ for
every $j$. Then
\[
  S\cdot\frac{\norm{a_r}}{\sigma_r}
   =\Big(\sum_{j=1}^m\sigma_j\norm{a_j}\Big)\frac{\norm{a_r}}{\sigma_r}
   =\sum_{j=1}^m\sigma_j\norm{a_j}\cdot\frac{\norm{a_r}}{\sigma_r}
   \;\ge\;\sum_{j=1}^m\sigma_j\norm{a_j}\cdot\frac{\norm{a_j}}{\sigma_j}
   =\sum_{j=1}^m\norm{a_j}^2=\norm{\mathbf{A}}_F^2,
\]
where the inequality replaces $\norm{a_r}/\sigma_r$ by the no-larger $\norm{a_j}/\sigma_j$ in each
term. Therefore
\[
  \smax(\mathbf{W})=\frac{\alpha\,S}{\norm{\mathbf{A}}_F^2}\cdot\frac{\norm{a_r}}{\sigma_r}
   \;\ge\;\frac{\alpha}{\norm{\mathbf{A}}_F^2}\cdot\norm{\mathbf{A}}_F^2=\alpha.
\]

\paragraph{Conclusion.}
The top eigenvalue of $\mathbf{AA}^\top$ never exceeds its trace, $\smax^2(\mathbf{A})\le\Tr(\mathbf{AA}^\top)=\norm{\mathbf{A}}_F^2$,
so $\alpha\ge\alpha\,\smax^2(\mathbf{A})/\norm{\mathbf{A}}_F^2$. Combining,
\[
  \smax(\mathbf{W})\;\ge\;\alpha\;\ge\;\frac{\alpha\,\smax^2(\mathbf{A})}{\norm{\mathbf{A}}_F^2},
\]
which is exactly the condition of Proposition~\ref{prop:taumono}(ii). The condition therefore
holds for every $\mathbf{A}$ and every noise profile, and the slack between the two sides is at least the
factor $\norm{\mathbf{A}}_F^2/\smax^2(\mathbf{A})$, the stable rank of $\mathbf{A}$. \hfill$\square$

\section{Detailed computation for Lemma~\ref{lem:descent}}\label{app:lemma-details}
This appendix gives, in full, the two expectations appearing in the proof of
Lemma~\ref{lem:descent}: the inner-product term $\EE[\inprod{d_k}{\hatx-x_k}\mid x_k]$ and the
squared-norm term $\EE[\norm{d_k}^2\mid x_k]$. Throughout, the conditional expectation
$\EE[\,\cdot\mid x_k]$ is taken over both the batch $\Bcal_k$ (sampled i.i.d.\ from $p$) and the
fresh noise; we drop the conditioning from the notation. We write $r_k=\mathbf{A}x_k-b$ for the residual,
so that $(r_k)_i=\inprod{a_i}{x_k}-b_i$, and recall from~\eqref{eq:noise-model} that
$\tilde b_i=b_i+\varepsilon_i$ with $\Enoise[\varepsilon_i]=0$, $\Enoise[\varepsilon_i^2]=\sigma_i^2$, the
$\varepsilon_i$ independent across $i$ and of the index draw.

\subsection*{B.1\quad Preliminaries: two expectation identities under i.i.d.\ sampling }
Let $g:\{1,\dots,m\}\to\RR^n$ be any fixed (deterministic given $x_k$) vector-valued function of
the row index, and let $I_1,\dots,I_\tau$ be the i.i.d.\ indices of the batch, each with
$\PP[I_j=i]=p_i$. We record two elementary facts used repeatedly below.

\emph{(a) Mean of a single draw.} For one index $I\sim p$,
\begin{equation}\label{eq:appB-mean}
  \Eidx\,[g(I)]=\sum_{i=1}^m p_i\,g(i).
\end{equation}

\emph{(b) Mean of an average of $\tau$ i.i.d.\ draws.} Put
$S=\frac1\tau\sum_{j=1}^\tau g(I_j)$, where $I_1,\dots,I_\tau\overset{\text{iid}}{\sim}p$. By
linearity, since each $I_j$ has the same marginal law $p$,
\begin{equation}\label{eq:appB-average}
  \Eidx[S]=\frac1\tau\sum_{j=1}^\tau \Eidx[g(I_j)]
          =\Eidx[g(I_1)]
          =\sum_{i=1}^m p_i\,g(i).
\end{equation}

\emph{(c) Second moment of an average of $\tau$ i.i.d.\ draws.} Put
$S=\frac1\tau\sum_{j=1}^\tau g(I_j)$. Expanding the squared norm and separating diagonal from
off-diagonal index pairs,
\begin{align}
  \Eidx[\norm{S}^2]
   &=\frac1{\tau^2}\,\Eidx\Big[\sum_{j=1}^\tau\norm{g(I_j)}^2
        +\sum_{\substack{j,l=1\\ j\ne l}}^\tau\inprod{g(I_j)}{g(I_l)}\Big]\nonumber\\
   &=\frac1{\tau^2}\Big[\tau\,\Eidx\norm{g(I)}^2
        +\tau(\tau-1)\,\big\langle\Eidx [g(I)],\,\Eidx [g(I)]\big\rangle\Big]\nonumber\\
   &=\frac1\tau\,\Eidx\norm{g(I)}^2
        +\frac{\tau-1}{\tau}\,\norm{\Eidx [g(I)]}^2,
  \label{eq:appB-secondmoment}
\end{align}
where the third line uses that there are $\tau$ diagonal terms (each with the same distribution as
a single draw) and $\tau(\tau-1)$ off-diagonal pairs, and that for $j\ne l$ the indices $I_j,I_l$
are independent, so $\Eidx[\inprod{g(I_j)}{g(I_l)}]=\inprod{\Eidx [g(I_j)]}{\Eidx [g(I_l)]}
=\norm{\Eidx [g(I)]}^2$.

\subsection*{B.2\quad The inner-product term}
By the definition~\eqref{eq:dk} of $d_k$ and using $\tilde b_i=b_i+\varepsilon_i$ together with
$\inprod{a_i}{\hatx}=b_i$, we first expand the inner product for a fixed batch. For a
single sampled index $i$,
\begin{align*}
  \Big\langle w_i\frac{\inprod{a_i}{x_k}-\tilde b_i}{\norm{a_i}^2}a_i,\ \hatx-x_k\Big\rangle
   =w_i\frac{(r_k)_i-\varepsilon_i}{\norm{a_i}^2}\,\inprod{a_i}{\hatx-x_k}
   =-w_i\frac{(r_k)_i-\varepsilon_i}{\norm{a_i}^2}\,(r_k)_i,
\end{align*}
since $\inprod{a_i}{\hatx-x_k}=\inprod{a_i}{\hatx}-\inprod{a_i}{x_k}=b_i-\inprod{a_i}{x_k}=-(r_k)_i$.
Averaging over the batch,
\begin{align*}
  \inprod{d_k}{\hatx-x_k}
   =-\frac1\tau\sum_{i\in\Bcal_k} w_i\frac{(r_k)_i^2}{\norm{a_i}^2}
    +\frac1\tau\sum_{i\in\Bcal_k} w_i\frac{\varepsilon_i\,(r_k)_i}{\norm{a_i}^2}.
\end{align*}
Taking the conditional expectation $\EE[\,\cdot\mid x_k]=\Eidx\Enoise$, the second sum vanishes because $\Enoise[\varepsilon_i]=0$ and $\varepsilon_i$ is
independent of the index draw and of $x_k$. For the first sum, applying~\eqref{eq:appB-mean} with
$g(i)=w_i(r_k)_i^2/\norm{a_i}^2$ (a scalar here) and that the $\tau$ draws are identically
distributed,
\begin{align*}
  \EE[\inprod{d_k}{\hatx-x_k}\mid x_k]
   =-\Eidx\Big[w_i\frac{(r_k)_i^2}{\norm{a_i}^2}\Big]
   =-\sum_{i=1}^m p_i\,\frac{w_i}{\norm{a_i}^2}\,(r_k)_i^2.
\end{align*}
Writing this as a quadratic form, with $e_i$ the canonical basis vectors and
$(r_k)_i=e_i^\top r_k$,
\begin{align*}
  \EE[\inprod{d_k}{\hatx-x_k}\mid x_k]
   =-\,r_k^\top\Big(\sum_{i=1}^m\frac{p_iw_i}{\norm{a_i}^2}\,e_ie_i^\top\Big)r_k
   =-\,r_k^\top\,\mathbf{PWD}^{-2}\,r_k,
\end{align*}
where we used $\sum_i \frac{p_iw_i}{\norm{a_i}^2}e_ie_i^\top
=\mathbf{PWD}^{-2}$. Finally, the coupling~\eqref{eq:coupling},
$\mathbf{PWD}^{-2}=\frac{\alpha}{\norm{\mathbf{A}}_F^2}\mathbf{I}$, gives
\begin{equation}\label{eq:appB-innerprod}
  \EE[\inprod{d_k}{\hatx-x_k}\mid x_k]
   =-\frac{\alpha}{\norm{\mathbf{A}}_F^2}\,r_k^\top r_k
   =-\frac{\alpha}{\norm{\mathbf{A}}_F^2}\,\norm{\mathbf{A}x_k-b}_2^2.
\end{equation}

\subsection*{B.3\quad The squared-norm term: splitting off the noise}
Decompose $d_k=d_k^{\mathrm{nf}}-d_k^{\mathrm{n}}$ into a noiseless part and a noise part,
\begin{align*}
  d_k=\frac1\tau\sum_{i\in\Bcal_k} w_i\,
      \frac{\inprod{a_{i}}{x_k}-\tilde b_i}{\norm{a_{i}}^2}\,a_{i} = \frac1\tau\sum_{i\in\Bcal_k} w_i\,
      \frac{(r_k)_i - \varepsilon_i}{\norm{a_{i}}^2}\,a_{i} = \frac1\tau\sum_{i\in\Bcal_k}v_i - \frac1\tau\sum_{i\in\Bcal_k}w_i\frac{\varepsilon_i}{\norm{a_i}^2}a_i,
\end{align*}
with
\begin{align*} d_k^{\mathrm{nf}}=\frac1\tau\sum_{i\in\Bcal_k}v_i,\quad
  v_i =w_i\frac{(r_k)_i}{\norm{a_i}^2}a_i,\qquad
  d_k^{\mathrm{n}}=\frac1\tau\sum_{i\in\Bcal_k}w_i\frac{\varepsilon_i}{\norm{a_i}^2}a_i,
\end{align*}
obtained by substituting $\tilde b_i=b_i+\varepsilon_i$ into~\eqref{eq:dk} and using
$\inprod{a_i}{x_k}-b_i=(r_k)_i$. Expanding the squared norm,
\begin{align*} 
  \EE[\norm{d_k}^2\mid x_k]
   =\Eidx\Enoise[\norm{d_k^{\mathrm{nf}}}^2]-2\,\Eidx\Enoise\inprod{d_k^{\mathrm{nf}}}{d_k^{\mathrm{n}}}
     +\Eidx\Enoise[\norm{d_k^{\mathrm{n}}}^2].
\end{align*}
The cross-term vanishes: with the batch fixed, the only randomness in $d_k^{\mathrm{n}}$ is the
noise, which is zero-mean, so the noise expectation alone gives
\begin{align*} 
  \Enoise[\inprod{d_k^{\mathrm{nf}}}{d_k^{\mathrm{n}}}]
   =\frac1{\tau^2}\sum_{i,l\in\Bcal_k}\frac{w_iw_l(r_k)_i}{\norm{a_i}^2\norm{a_l}^2}
      \,\Enoise[\varepsilon_l]\,\inprod{a_i}{a_l}=0.
\end{align*}
Since $d_k^{\mathrm{nf}}$ does not depend on the noise, $\Eidx\Enoise[\norm{d_k^{\mathrm{nf}}}^2]
=\Eidx[\norm{d_k^{\mathrm{nf}}}^2]$, and we treat the two surviving parts separately.

\subsubsection*{B.3.1\quad The noiseless part $d_k^{\mathrm{nf}}$}
Apply~\eqref{eq:appB-secondmoment} with $g(i)=v_i=w_i(r_k)_i a_i/\norm{a_i}^2$:
\begin{equation}\label{eq:appB-nf-split}
  \Eidx[\norm{d_k^{\mathrm{nf}}}^2]
   =\frac1\tau\,\Eidx[\norm{v_i}^2]+\frac{\tau-1}{\tau}\,\norm{\Eidx [v_i]}^2.
\end{equation}
We compute each moment as a quadratic form in $r_k$.

\noindent \emph{First moment (diagonal term).} Since $\norm{a_i}^2=a_i^\top a_i$,
\[
  \norm{v_i}^2=w_i^2\frac{(r_k)_i^2}{\norm{a_i}^4}\,\norm{a_i}^2
             =w_i^2\frac{(r_k)_i^2}{\norm{a_i}^2},
\]
so by~\eqref{eq:appB-mean} and using $\sum_i \frac{p_iw_i^2}{\norm{a_i}^2}e_ie_i^\top=\mathbf{P W}^2 \mathbf{D}^{-2}$,
\begin{align*}
  \Eidx[\norm{v_i}^2]
   =\sum_{i=1}^m p_i\,\frac{w_i^2}{\norm{a_i}^2}\,(r_k)_i^2
   =r_k^\top\Big(\sum_{i=1}^m\frac{p_iw_i^2}{\norm{a_i}^2}e_ie_i^\top\Big)r_k
   =\big\langle r_k,\,\mathbf{PW}^2\mathbf{D}^{-2}r_k\big\rangle,
\end{align*}
\noindent \emph{Second moment (off-diagonal term).} Using $(r_k)_i=e_i^\top r_k=e_i^\top(\mathbf{A}x_k-b)$ and
$a_i=\mathbf{A}^\top e_i$,
\begin{align*}
  \Eidx [v_i]
   =\sum_{i=1}^m p_i\,w_i\frac{(r_k)_i}{\norm{a_i}^2}a_i
   &=\sum_{i=1}^m \frac{p_iw_i}{\norm{a_i}^2}\,\mathbf{A}^\top e_i\,(e_i^\top r_k) \\
   &=\mathbf{A}^\top\Big(\sum_{i=1}^m\frac{p_iw_i}{\norm{a_i}^2}e_ie_i^\top\Big)r_k \\
   &=\mathbf{A}^\top \mathbf{PWD}^{-2}r_k.
\end{align*}
Therefore
\begin{align*}
  \norm{\Eidx [v_i]}^2
   =\big(\mathbf{A}^\top \mathbf{PWD}^{-2}r_k\big)^\top\big(\mathbf{A}^\top \mathbf{PWD}^{-2}r_k\big)
&=r_k^\top\,\mathbf{PWD}^{-2}\,\mathbf{A A}^\top\,\mathbf{PWD}^{-2}\,r_k \\
   &=\big\langle r_k,\,\mathbf{PWD}^{-2}\mathbf{AA}^\top \mathbf{PWD}^{-2}r_k\big\rangle,
\end{align*}
using $(\mathbf{PWD}^{-2})^\top=\mathbf{PWD}^{-2}$ since diagonal matrices are symmetric.

\noindent \emph{Assembling the bracket.} Substituting both moments into~\eqref{eq:appB-nf-split},
\begin{equation}\label{eq:appB-bracket}
  \Eidx[\norm{d_k^{\mathrm{nf}}}^2]
   =\Big\langle r_k,\ \underbrace{\Big(\tfrac1\tau\,\mathbf{PW}^2\mathbf{D}^{-2}
       +\tfrac{\tau-1}\tau\,\mathbf{PWD}^{-2}\mathbf{AA}^\top \mathbf{PWD}^{-2}\Big)}_{=:\,\mathbf{Q}}\ r_k\Big\rangle.
\end{equation}
We now simplify $\mathbf{Q}$ using the coupling~\eqref{eq:coupling}. In the second summand, replace
\emph{one} factor $\mathbf{PWD}^{-2}$ by $(\alpha/\norm{\mathbf{A}}_F^2)\mathbf{I}$ and leave the other, and in the first
summand write $\mathbf{PW}^2\mathbf{D}^{-2}=\mathbf{W}\cdot \mathbf{PWD}^{-2}$ and replace $\mathbf{PWD}^{-2}$ likewise:
\begin{align*}  \mathbf{PW}^2\mathbf{D}^{-2}=\mathbf{W}\,\mathbf{PWD}^{-2}=\frac{\alpha}{\norm{\mathbf{A}}_F^2}\,\mathbf{W}, 
\end{align*}
and
\begin{align*} 
\mathbf{PWD}^{-2}\mathbf{AA}^\top \mathbf{PWD}^{-2}=\frac{\alpha}{\norm{\mathbf{A}}_F^2}\,\mathbf{PWD}^{-2}\mathbf{AA}^\top
   =\frac{\alpha^2}{\norm{\mathbf{A}}_F^4}\,\mathbf{AA}^\top.
\end{align*}
Hence
\begin{align*} 
  \mathbf{Q}=\frac1\tau\cdot\frac{\alpha}{\norm{\mathbf{A}}_F^2}\,\mathbf{W}
     +\frac{\tau-1}{\tau}\cdot\frac{\alpha^2}{\norm{\mathbf{A}}_F^4}\,\mathbf{AA}^\top
   =\frac{\alpha}{\norm{\mathbf{A}}_F^2}\Big(\frac{1}{\tau}\,\mathbf{W}
       +\frac{\alpha}{\norm{\mathbf{A}}_F^2}\,\frac{\tau-1}{\tau}\,\mathbf{AA}^\top\Big)
   =\frac{2\alpha}{\norm{\mathbf{A}}_F^2}\,\mathbf{T},
\end{align*}
where the last equality is the definition~\eqref{eq:Tdef} of
$\mathbf{T}=\tfrac{1}{2\tau}\mathbf{W}+\tfrac{\alpha}{2\norm{\mathbf{A}}_F^2}(1-\tfrac1\tau)\mathbf{AA}^\top$ (factor $\tfrac12$ pulled
out). Substituting into~\eqref{eq:appB-bracket} and bounding the Rayleigh quotient by the largest
eigenvalue,
\begin{equation}\label{eq:appB-nf-final}
  \Eidx[\norm{d_k^{\mathrm{nf}}}^2]
   =\frac{2\alpha}{\norm{\mathbf{A}}_F^2}\,\langle r_k,\,\mathbf{T} r_k\rangle
   \le\frac{2\alpha}{\norm{\mathbf{A}}_F^2}\,\smax(\mathbf{T})\,\norm{r_k}_2^2,
\end{equation}
where $\smax(\mathbf{T})=\norm{\mathbf{T}}_2$ since $\mathbf{T}$ is symmetric positive semidefinite (it is a nonnegative
combination of the positive semidefinite matrices $\mathbf{W}$ and $\mathbf{AA}^\top$).

\subsubsection*{B.3.2\quad The noisy part $d_k^{\mathrm{n}}$}
Write $d_k^{\mathrm{n}}=\frac1\tau\sum_{i\in\Bcal_k}u_i$ with
$u_i:=w_i\varepsilon_i a_i/\norm{a_i}^2$. We compute the conditional second moment
$\Eidx\Enoise[\norm{d_k^{\mathrm{n}}}^2]$, taking the noise expectation $\Enoise$ first (with the
batch fixed) and the batch expectation $\Eidx$ second. Expanding the squared norm and separating
the diagonal from the off-diagonal index pairs,
\begin{align*}
  \Eidx\Enoise[\norm{d_k^{\mathrm{n}}}^2]  =\frac1{\tau^2}\,\Eidx\,\Enoise\Big[\sum_{j=1}^\tau\norm{u_{I_j}}^2
       +\sum_{\substack{j,l=1\\ j\ne l}}^\tau\inprod{u_{I_j}}{u_{I_l}}\Big].
\end{align*}
For an off-diagonal pair $j\ne l$, the indices $I_j,I_l$ are distinct draws so the noises
$\varepsilon_{I_j},\varepsilon_{I_l}$ are independent and zero-mean; the noise expectation alone
already annihilates the term,
\begin{align*}
  \Enoise[\inprod{u_{I_j}}{u_{I_l}}]
   =\frac{w_{I_j}w_{I_l}\,\Enoise[\varepsilon_{I_j}]\,\Enoise[\varepsilon_{I_l}]}
        {\norm{a_{I_j}}^2\norm{a_{I_l}}^2}\,\inprod{a_{I_j}}{a_{I_l}}=0,
\end{align*}
since $\Enoise[\varepsilon_{I_j}]=\Enoise[\varepsilon_{I_l}]=0$. Only the $\tau$ diagonal terms
survive. For each, $\Enoise[\norm{u_{I_j}}^2]
=w_{I_j}^2\,\Enoise[\varepsilon_{I_j}^2]\,\norm{a_{I_j}}^2/\norm{a_{I_j}}^4
=w_{I_j}^2\sigma_{I_j}^2/\norm{a_{I_j}}^2$, using $\Enoise[\varepsilon_{I_j}^2]=\sigma_{I_j}^2$.
Hence
\begin{align*}
  \Eidx\Enoise[\norm{d_k^{\mathrm{n}}}^2]
   =\frac1{\tau^2}\cdot\tau\,\Eidx\Big[\frac{w_I^2\sigma_I^2}{\norm{a_I}^2}\Big]
   =\frac1\tau\,\Eidx\Big[\frac{w_I^2\sigma_I^2}{\norm{a_I}^2}\Big].
\end{align*}
Applying the batch-mean identity~\eqref{eq:appB-mean},
\begin{equation}\label{eq:appB-noise-final}
  \Eidx\Enoise[\norm{d_k^{\mathrm{n}}}^2]
   =\frac1\tau\sum_{i=1}^m\frac{p_iw_i^2\sigma_i^2}{\norm{a_i}^2}
   =\frac1\tau\,\Tr\!\big(\mathbf{PW}^2\mathbf{\Sigma D}^{-2}\big),
\end{equation}
where the last equality recognises $\sum_i p_iw_i^2\sigma_i^2/\norm{a_i}^2$ as the trace of the
diagonal matrix \newline $\mathbf{PW}^2\mathbf{\Sigma D}^{-2}=\Diag\big(p_iw_i^2\sigma_i^2/\norm{a_i}^2\big)$.

\subsection*{B.4\quad Assembling the one-step bound}
Insert the inner-product identity~\eqref{eq:appB-innerprod} and the two squared-norm
results~\eqref{eq:appB-nf-final}--\eqref{eq:appB-noise-final} into the deterministic descent
inequality of Lemma~\ref{lem:onestep}, $\Df{x_{k+1}^*}{x_{k+1}}{\hatx}\le\Df{x_k^*}{x_k}{\hatx}
+\eta_k\inprod{d_k}{\hatx-x_k}+\tfrac{\eta_k^2}{2}\norm{d_k}^2$, and take the \emph{conditional}
expectation $\EE[\,\cdot\mid x_k]=\Eidx\Enoise$ (under which $\Df{x_k^*}{x_k}{\hatx}$ and
$\norm{\mathbf{A}x_k-b}_2^2$ are constant):
\begin{align*}
  \EE\big[\Df{x_{k+1}^*}{x_{k+1}}{\hatx}\mid x_k\big]
   &\le\Df{x_k^*}{x_k}{\hatx}
     -\eta_k\frac{\alpha}{\norm{\mathbf{A}}_F^2}\norm{\mathbf{A}x_k-b}_2^2\\
   &\qquad+\frac{\eta_k^2}{2}\Big(\frac{2\alpha}{\norm{\mathbf{A}}_F^2}\smax(\mathbf{T})\norm{\mathbf{A}x_k-b}_2^2
       +\frac1\tau\Tr(\mathbf{PW}^2\mathbf{\Sigma D}^{-2})\Big)\\
   &=\Df{x_k^*}{x_k}{\hatx}
     -\frac{\alpha\eta_k}{\norm{\mathbf{A}}_F^2}\big(1-\eta_k\smax(\mathbf{T})\big)\norm{\mathbf{A}x_k-b}_2^2
     +\frac{\eta_k^2}{2\tau}\Tr(\mathbf{PW}^2\mathbf{\Sigma D}^{-2}),
\end{align*}
where we used $\Eidx\Enoise[\inprod{d_k}{\hatx-x_k}]$ from~\eqref{eq:appB-innerprod},
$\Eidx[\norm{d_k^{\mathrm{nf}}}^2]$ from~\eqref{eq:appB-nf-final}, and
$\Eidx\Enoise[\norm{d_k^{\mathrm{n}}}^2]$ from~\eqref{eq:appB-noise-final}. Now apply the error bound
(Assumption~\ref{ass:errorbound}), $\norm{\mathbf{A}x_k-b}_2^2\ge\theta(\hatx)\,\Df{x_k^*}{x_k}{\hatx}$,
with $\gamma=\frac{\theta(\hatx)}{\norm{\mathbf{A}}_F^2}$. The coefficient
$\tfrac{\alpha\eta_k}{\norm{\mathbf{A}}_F^2}(1-\eta_k\smax(\mathbf{T}))$ is nonnegative for the admissible step
$\eta_k\le1/\smax(\mathbf{T})$, so multiplying the residual lower bound through preserves the inequality
and gives the conditional bound
\[
  \EE\big[\Df{x_{k+1}^*}{x_{k+1}}{\hatx}\mid x_k\big]
   \le\big(1-\gamma\alpha\eta_k(1-\eta_k\smax(\mathbf{T}))\big)\,\Df{x_k^*}{x_k}{\hatx}
     +\frac{\eta_k^2}{2\tau}\Tr(\mathbf{PW}^2\mathbf{\Sigma D}^{-2}).
\]
Finally, taking the total expectation $\Etot$ of both sides and using the tower rule
$\Etot[\,\cdot\,]=\Etot\big[\EE[\,\cdot\mid x_k]\big]$ together with the linearity of $\Etot$
yields
\[
  \Etot\,[\Df{x_{k+1}^*}{x_{k+1}}{\hatx}]
   \le\big(1-\gamma\alpha\eta_k(1-\eta_k\smax(\mathbf{T}))\big)\,\Etot\,[\Df{x_k^*}{x_k}{\hatx}]
     +\frac{\eta_k^2}{2\tau}\Tr(\mathbf{PW}^2\mathbf{\Sigma D}^{-2}),
\]
which is exactly the statement of Lemma~\ref{lem:descent}. \hfill$\square$

\section{Spectral bounds: proof of Lemma~\ref{lem:Tbound}}\label{app:spectral}
A proof of this Lemma can be found in~\cite[Lemma 4.10]{tondji2023faster} but we include the proof here for completeness. For the upper bound, the spectral-norm triangle inequality applied to the two summands of
$\mathbf{T}=\tfrac1{2\tau}\mathbf{W}+\tfrac{\alpha}{2\norm{\mathbf{A}}_F^2}(1-\tfrac1\tau)\mathbf{AA}^\top$ gives
\[
  \smax(\mathbf{T})\le\frac1{2\tau}\smax(\mathbf{W})+\frac{\alpha}{2\norm{\mathbf{A}}_F^2}\Big(1-\frac1\tau\Big)\smax^2(\mathbf{A}),
\]
since $\smax(\mathbf{AA}^\top)=\smax^2(\mathbf{A})$. For $\mathbf{W}=\alpha \mathbf{I}$ the two summands are simultaneously
diagonalised by the eigenbasis of $\mathbf{AA}^\top$: writing $\mathbf{AA}^\top=\sum_\ell\sigma_\ell^2 u_\ell
u_\ell^\top$,
\[
  \mathbf{T}=\sum_\ell\Big(\frac{\alpha}{2\tau}
     +\frac{\alpha}{2\norm{\mathbf{A}}_F^2}\Big(1-\frac1\tau\Big)\sigma_\ell^2\Big)u_\ell u_\ell^\top,
\]
so the eigenvalues of $\mathbf{T}$ are increasing affine functions of $\sigma_\ell^2$ and the largest is
attained at $\sigma_\ell^2=\smax^2(\mathbf{A})$, giving the stated identity
$\smax(\mathbf{T})=\tfrac{\alpha}{2\tau}\big(1+(\tau-1)\smax^2(\mathbf{A})/\norm{\mathbf{A}}_F^2\big)$. The monotonicity in
$\tau$ and the limit then follow as in the proof of Proposition~\ref{prop:taumono} in
\S\ref{sec:spectral}. \hfill$\square$

\section{Full proof of Theorem~\ref{thm:main}}\label{app:thm-proof}
We give the complete argument abbreviated in \S\ref{sec:thmproof}, following the structure of the
one-step bound (Lemma~\ref{lem:descent}). To obtain a rate $q\in(0,1)$ we need $1-\eta_k\smax(\mathbf{T})>0$, that is $\eta_k<1/\smax(\mathbf{T})$. From
Lemma~\ref{lem:descent}, it holds:
\begin{equation}\label{eq:app-recurrence}
  \Etot\,[\Df{x_{k+1}^*}{x_{k+1}}{\hatx}]
   \le\big(1-\alpha\gamma\eta_k(1-\eta_k\smax(\mathbf{T}))\big)\,\Etot\,[\Df{x_k^*}{x_k}{\hatx}]
     +\frac{\eta_k^2}{2\tau}\Tr(\mathbf{PW}^2\mathbf{\Sigma D}^{-2}).
\end{equation}

\paragraph{Step 1: unrolling the recurrence.}
Let $\eta=(\eta_0,\eta_1,\dots)$ collect the step sizes, and abbreviate
$\rho_j:=1-\alpha\gamma\eta_j(1-\eta_j\smax(T))$. Iterating~\eqref{eq:app-recurrence} from $0$ to
$k$ and using that $\Etot\,[\Df{x_0^*}{x_0}{\hatx}]=\Df{x_0^*}{x_0}{\hatx}$ is constant, we obtain
\begin{align}
  \Etot\,[\Df{x_{k+1}^*}{x_{k+1}}{\hatx}]
   &\le\prod_{j=0}^{k}\rho_j\;\Etot\,[\Df{x_0^*}{x_0}{\hatx}]
     +\frac{\Tr(\mathbf{PW}^2\mathbf{\Sigma D}^{-2})}{\tau}\sum_{j=0}^{k}\frac{\eta_j^2}{2}\prod_{i=j+1}^{k}\rho_i,
   \label{eq:app-51}
\end{align}
with the convention that the empty product $\prod_{i=k+1}^{k}\rho_i=1$. Dividing
through by $\tfrac1\tau\Tr(\mathbf{PW}^2\mathbf{\Sigma D}^{-2})$ gives the closed form of the error sequence,
\begin{equation}\label{eq:app-52}
  D_k(\eta)=\prod_{j=0}^{k}\rho_j\;\frac{\tau\,\Df{x_0^*}{x_0}{\hatx}}{\Tr(\mathbf{PW}^2\mathbf{\Sigma D}^{-2})}
            +\sum_{j=0}^{k}\frac{\eta_j^2}{2}\prod_{i=j+1}^{k}\rho_i,
\end{equation}
so that~\eqref{eq:app-recurrence} reads
\begin{equation}\label{eq:app-Dk}
  \Etot\,[\Df{x_{k+1}^*}{x_{k+1}}{\hatx}]\le D_k(\eta)\cdot\frac{1}{\tau}\Tr(\mathbf{PW}^2\mathbf{\Sigma D}^{-2}),
\end{equation}
which is \eqref{eq:app-51} restated. From eq~\eqref{eq:app-Dk}, It holds that $D_k(\eta)$ satisfies the recurrence
\begin{equation}\label{eq:app-Dk1}
  D_k(\eta)=\big(1-\alpha\gamma\eta_k(1-\eta_k\smax(\mathbf{T}))\big)D_{k-1}(\eta)+\frac{\eta_k^2}{2}.
\end{equation}

\paragraph{Step 2: the optimal step size.}
Using the empty-product convention, $D_k(\eta)$ obeys the recurrence~\eqref{eq:app-Dk1}.
Differentiating with respect to the single free variable $\eta_k$,
\[
  \frac{\partial D_k(\eta)}{\partial\eta_k}
   =\big(-\alpha\gamma+2\alpha\gamma\eta_k\smax(\mathbf{T})\big)D_{k-1}(\eta)+\eta_k=0,
\]
which gives the optimal adaptive step size
\begin{equation}\label{eq:app-54}
  \eta_k=\frac{\alpha\gamma D_{k-1}(\eta)}{1+2\alpha\gamma\smax(T)D_{k-1}(\eta)}.
\end{equation}
Since $\partial^2 D_k(\eta)/\partial\eta_k^2=2\alpha\gamma\smax(\mathbf{T})D_{k-1}(\eta)+1>0$, the value
\eqref{eq:app-54} minimises $D_k(\eta)$. As $D_{k-1}(\eta)$ is unavailable in closed form, we
estimate it, following~\cite{marshall2023optimal,tondji2024adaptive}, by the deterministic
sequence $\beta_k:=D_{k-1}(\eta)$ with
$\beta_0=\tau\,\Df{x_0^*}{x_0}{\hatx}/\Tr(\mathbf{PW}^2\mathbf{\Sigma D}^{-2})$, giving
\begin{equation}\label{eq:app-55}
  \eta_k=\frac{\alpha\gamma\beta_k}{1+2\alpha\gamma\smax(\mathbf{T})\beta_k}.
\end{equation}

\paragraph{Step 3: admissibility.}
The requirement $\eta_k<1/\smax(\mathbf{T})$ holds automatically:
\begin{align*}
  \frac{\alpha\gamma\beta_k}{1+2\alpha\gamma\smax(\mathbf{T})\beta_k}<\frac{1}{\smax(\mathbf{T})}
  &\iff \alpha\gamma\smax(\mathbf{T})\beta_k<1+2\alpha\gamma\smax(\mathbf{T})\beta_k \\
  &\iff 0<1+\alpha\gamma\smax(\mathbf{T})\beta_k,
\end{align*}
which is always true since $\alpha,\gamma,\smax(\mathbf{T}),\beta_k\ge0$.

\paragraph{Step 4: telescoping the auxiliary sequence.}
Substituting~\eqref{eq:app-55} into~\eqref{eq:app-Dk1} and using the identity
$\eta_k\big(1+2\alpha\gamma\smax(\mathbf{T})\beta_k\big)=\alpha\gamma\beta_k$,
\begin{align*}
  \beta_{k+1}
   &=\big(1-\alpha\gamma\eta_k(1-\eta_k\smax(\mathbf{T}))\big)\beta_k+\frac{\eta_k^2}{2}\\
   &=\beta_k-\alpha\gamma\beta_k\eta_k+\frac{\eta_k^2}{2}\big(1+2\alpha\gamma\smax(\mathbf{T})\beta_k\big)\\
   &=\beta_k-\alpha\gamma\beta_k\eta_k+\frac{\eta_k^2}{2}\cdot\frac{\alpha\gamma\beta_k}{\eta_k}\\
   &=\beta_k-\alpha\gamma\beta_k\eta_k+\frac{\alpha\gamma\beta_k\eta_k}{2} \\
    &=\beta_k\Big(1-\frac{\alpha\gamma\eta_k}{2}\Big).
\end{align*}
Thus $\beta_k$ is deterministic and strictly decreasing and since $\eta_k>0$ for all $k$, it
converges to zero. 

\paragraph{Step 5: conclusion.}
Combining~\eqref{eq:app-Dk} with the identification $\beta_k=D_{k-1}(\eta)$ and the
$1$-strong-convexity inequality $\norm{x_k-\hatx}_2^2\le 2\,\Df{x_k^*}{x_k}{\hatx}$ yields
\[
  \Etot\,[\norm{x_k-\hatx}_2^2]\le 2\,\Etot\,[\Df{x_k^*}{x_k}{\hatx}]
   \le\frac{2\,\Tr(\mathbf{PW}^2\mathbf{\Sigma D}^{-2})}{\tau}\,\beta_k,
\]
which is~\eqref{eq:main-rate}. \hfill$\square$

\section{Proof of Corollary~\ref{cor:asymp}}
\label{sec:app-cor-asymp}
We start this section by giving a short background on the Lambert-$W$ Function. The Lambert-$W$ function is the multivalued inverse of the complex function $w \mapsto w e^w$. The principal branch, denoted $W_0(x)$, is real-valued, strictly increasing, and concave on $x \in [-1/e, \infty)$, satisfying $W_0(x) e^{W_0(x)} = x$ and detailing its small-argument Taylor expansion $W_0(x) = x - x^2 + \mathcal{O}(x^3)$ and its large-argument asymptotic series $W_0(x) = \ln x - \ln \ln x + o(1)$.\\

The recursion~(\ref{eq:etaopt}) within the expression of $\eta_k$ can be written as 
\begin{align*}
    \dfrac{\gamma\beta_{k+1}-\gamma\beta_{k}}{\gamma}&=\beta_{k+1}-\beta_{k}=-\dfrac{\alpha\gamma\beta_k}{2}\eta_k=-\dfrac{1}{2}\; \dfrac{(\alpha\gamma \beta_k)^2}{1 +2\alpha\gamma\sigma_{\max}\mathbf{(T)}\beta_k}.
\end{align*}
Setting $u_k:=\gamma\beta_k$, we obtain that
\begin{align*}
    \dfrac{u_{k+1}-u_k}{\gamma}&=-\dfrac{\alpha^2}{2}\dfrac{u_k^2}{1 +2\alpha\sigma_{\max}\mathbf{(T)}u_k}=-\dfrac{\alpha}{4\sigma_{\max}\mathbf{(T)}}\dfrac{u_k^2}{u_k + \dfrac{1}{2\alpha\sigma_{\max}\mathbf{(T)}}}.
\end{align*}
This new recursion is  a forward-Euler approximation to the ordinary differential equation $$\dot{u}=-\dfrac{\alpha}{4\sigma_{\max}\mathbf{(T)}}\dfrac{u^2}{u + \dfrac{1}{2\alpha\sigma_{\max}\mathbf{(T)}}}.$$
It is easy to verify that the solution of this differential equation is
$$u(t)=\dfrac{1}{2\alpha\sigma_{\max}\mathbf{(T)}W_0\left(c_0\exp\left(\dfrac{\alpha t}{4\sigma_{\max}\mathbf{(T)}}\right)\right)}\quad \text{where}\quad c_0=\dfrac{1}{2\alpha\sigma_{\max}\mathbf{(T)}u_0}\exp\left(\dfrac{1}{2\alpha\sigma_{\max}\mathbf{(T)}u_0}\right),$$
is the initial condition and $W_0$ is the principal branch of the Lambert-$W$ function. In addition, $u$ is a convex function, because if we set $c_1=\alpha/\left(4\sigma_{\max}\mathbf{(T)}\right)$ and $c_2=1/\left(2\alpha\sigma_{\max}\mathbf{(T)}\right)$, it is easy to see that 
$\ddot{u}(t)=c_1^2u^3(u+2c_2)/(u+c_2)^3 \geq 0$ when $u_0:=u(0)\geqslant 0$. We then deduce that the error decreases rapidly at first and then slows down as it approaches the asymptote. Since the forward Euler method is a lower bound for convex functions, this means that
\begin{align*}
    \beta_k &\leq  \dfrac{1}{2\alpha\gamma\sigma_{\max}\mathbf{(T)}W_0\left(c_0\exp\left(\dfrac{\alpha k}{4\sigma_{\max}\mathbf{(T)}}\right)\right)}\quad \text{with}\quad c_0=\dfrac{1}{2\alpha\gamma\sigma_{\max}(T)\beta_0}\exp\left(\dfrac{1}{2\alpha\gamma\sigma_{\max}\mathbf{(T)}\beta_0}\right). \label{beta} 
\end{align*}
Now, let denote by $h(k) = \dfrac{\alpha}{\tau}\cdot\Tr(\mathbf{W\Sigma}) \cdot \beta_k$ so that by~\eqref{eq:main-rate},
$\Etot\,[\norm{x_k-\hatx}_2^2]\le \frac{2 h(k)}{\|\mathbf{A}\|_F^2}$. Define $\bar{\sigma}=\sum_{i=1}^{m}\sigma_i\|a_i\|$; then we have 
\begin{align*}
    \sigma \to 0 \Leftrightarrow \sigma_i \to 0 \,\, \forall\,i \Leftrightarrow \bar \sigma \to 0.
\end{align*}

\paragraph{Uniform weights.}
We have the following $\Tr(\mathbf{W\Sigma})=\alpha \sigma^2,\; \beta_0 = \frac{\tau\|\mathbf{A}\|_F^2 D_f^{x_0^*}(x_0, \hat{x})}{\alpha^2\sigma^2}$ and 
\begin{align*}
 h(k) \le  \dfrac{\alpha\sigma^2}{2\tau\gamma\sigma_{\max}(\mathbf{T})W_0\left(c_0\exp\left(\dfrac{\alpha\gamma k}{4\sigma_{\max}(\mathbf{T})}\right)\right)},
\end{align*}
where
\begin{footnotesize}
  \begin{align*}
    c_0&=\dfrac{\alpha\sigma^2}{2\gamma\tau\sigma_{\text{max}}(\mathbf{T})\|\mathbf{A}\|_F^2 D_f^{x_0^*}(x_0, \hat{x})}\exp\left(\dfrac{\alpha\sigma^2}{2\gamma\tau\sigma_{\text{max}}(\mathbf{T})\|\mathbf{A}\|_F^2 D_f^{x_0^*}(x_0, \hat{x})}\right).
\end{align*}  
\end{footnotesize}
Let define :$$c=2\gamma\tau\sigma_{\text{max}}(\mathbf{T})\|\mathbf{A}\|_F^2,\quad a=2\tau\gamma\sigma_{\max}(\mathbf{T}),\quad  \text{and}\quad v=\dfrac{\alpha\gamma }{4\sigma_{\max}(\mathbf{T})}.$$
  In the regime $\sigma \to 0$ the constant $c_0 \to 0$, so the argument $c_0 e^{vk} \to 0$ at any fixed $k$; the principal branch then admits the small-argument expansion  $W_0(x)=x+\mathcal{O}(x^2)$ as $x \rightarrow 0$. 
  Thus,
  \begin{align*}
   h(k)\le\dfrac{\alpha\sigma^2}{aW_0(c_0 e^{vk})}
   &=\dfrac{\alpha\sigma^2}{\dfrac{a\alpha \sigma^2}{cD_f^{x_0^*}(x_0, \hat{x})} e^{vk}e^{\dfrac{\alpha \sigma^2}{cD_f^{x_0^*}(x_0,\hat{x})}} \; +\mathcal{O}\left(\sigma^4\right)}\\
   &=\dfrac{cD_f^{x_0^*}(x_0, \hat{x})}{a}e^{-vk}\quad+\quad \mathcal{O}\left(\sigma^2\right)\\
   &=\|\mathbf{A}\|_F^2\exp\left(-\dfrac{\alpha\gamma k}{4\sigma_{\max}(\mathbf{T})}\right)D_f^{x_0^*}(x_0, \hat{x}) + \mathcal{O}\left(\sigma^2\right).
  \end{align*}

 \paragraph{Optimal noise-aware weights.}  
We have the following :
$\Tr(\mathbf{W\Sigma})=\dfrac{\alpha \bar\sigma^2}{\|\mathbf{A}\|_F^2},$
\begin{align*}
    h(k) &\le  \dfrac{\alpha\bar{\sigma}^2/\|\mathbf{A}\|_F^2}{2\tau\gamma\sigma_{\max}(\mathbf{T})W_0\left(c_0\exp\left(\dfrac{\alpha\gamma k}{4\sigma_{\max}(\mathbf{T})}\right)\right)}\quad\text{and}\quad \beta_0 = \frac{\tau\|\mathbf{A}\|_F^4 D_f^{x_0^*}(x_0, \hat{x})}{\alpha\bar{\sigma}^2},
\end{align*}
with
\begin{footnotesize}
  \begin{align*}
    c_0=\dfrac{\alpha\bar{\sigma}^2}{2\gamma\tau\sigma_{\text{max}}(\mathbf{T})\|\mathbf{A}\|_F^4 D_f^{x_0^*}(x_0, \hat{x})}\exp\left(\dfrac{\alpha\bar{\sigma}^2}{2\gamma\tau\sigma_{\text{max}}(\mathbf{T})\|\mathbf{A}\|_F^4 D_f^{x_0^*}(x_0, \hat{x})}\right).
\end{align*} 
\end{footnotesize}
    The proof is similar to the previous one. To see this, rewrite the upper bound of $h(k)$  as
    $$h(k) \le \dfrac{\alpha\bar{\sigma}^2/\|\mathbf{A}\|_F^2}{aW_0(c_0 e^{vk})}\quad \text{with}\quad c_0=\dfrac{\bar{\alpha\sigma}^2}{c D_f^{x_0^*}(x_0, \hat{x})}\exp\left(\dfrac{\bar{\alpha\sigma}^2}{c D_f^{x_0^*}(x_0, \hat{x})}\right), $$
where $c=2\gamma\tau\sigma_{\text{max}}(\mathbf{T})\|\mathbf{A}\|_F^4,\quad a=2\tau\gamma\sigma_{\max}(\mathbf{T})\quad  \text{and}\quad v=\dfrac{\alpha\gamma }{4\sigma_{\max}(\mathbf{T})}.$
Thus,
  \begin{align*}
   h(k)\le \dfrac{\alpha\bar{\sigma}^2/\|\mathbf{A}\|_F^2}{aW_0(c_0 e^{vk})}
   &=\dfrac{\alpha\bar{\sigma}^2/\|\mathbf{A}\|_F^2}{\dfrac{a\alpha \bar\sigma^2}{cD_f^{x_0^*}(x_0, \hat{x})} e^{vk}e^{\dfrac{\alpha \bar\sigma^2}{cD_f^{x_0^*}(x_0,\hat{x})}} \; +\mathcal{O}\left(\bar\sigma^4\right)}\\
   &=\dfrac{cD_f^{x_0^*}(x_0, \hat{x})}{a \|\mathbf{A}\|_F^2}e^{-vk}\quad+\; \mathcal{O}\left(\bar\sigma^2\right)\\
   &=\|\mathbf{A}\|_F^2\exp\left(-\dfrac{\alpha\gamma k}{4\sigma_{\max}(\mathbf{T})}\right)D_f^{x_0^*}(x_0, \hat{x}) \; + \mathcal{O}\left(\bar\sigma^2\right).
  \end{align*}
Now, we have a look at the case $k \to \infty$. The principal branch of the Lambert-$W_0$ function has an asymptotic expansion
\begin{align*}
    W_0(e^{k})&=k - ln(k) + \mathcal{O}\left(\dfrac{ln(k)}{k}\right)\quad \text{as} \; k \rightarrow \infty.
\end{align*}
For simplicity, define 
\begin{align*}
    K = \dfrac{\text{Tr}(\mathbf{W\Sigma})}{2\tau\gamma \sigma_{\max}(\mathbf{T})}\quad \text{and} \quad v=\dfrac{\alpha\gamma}{4 \sigma_{\max}(\mathbf{T})}.
\end{align*}
Plugging these into the upper bound of $h(k),$ we have
\begin{align*}
    h(k) \leq \dfrac{K}{vk\left(1+\mathcal{O}\left(\dfrac{ln(k)}{k}\right)\right)} =\frac{2\text{Tr}(\mathbf{W\Sigma})) }{\alpha\tau\gamma^2 k}\left(1+\mathcal{O}\left(\dfrac{ln(k)}{k}\right)\right) = \begin{cases}
        \frac{2\sigma^2 }{\tau\gamma^2 k} ,\quad \text{If}\,\, \mathbf{W} = \alpha \mathbf{I}\\\\
        \frac{2\bar\sigma^2 }{\tau k\gamma^2 \|\mathbf{A}\|_F^2}  ,\quad \text{Else}
    \end{cases},
\end{align*} 
so that by~\eqref{eq:main-rate},
\begin{align*}
\Etot\,[\norm{x_k-\hatx}_2^2]\le \frac{2h(k)}{\|\mathbf{A}\|_F^2} \leq \begin{cases}
        \frac{4\sigma^2 }{\tau k \gamma^2  \|\mathbf{A}\|_F^2} ,\quad \text{If}\,\, \mathbf{W} = \alpha \mathbf{I}\\\\
        \frac{4\bar\sigma^2 }{\tau k\gamma^2 \|\mathbf{A}\|_F^4}  ,\quad \text{Else}
    \end{cases} = \frac{4}{\tau \alpha^2 \gamma^2 k }\Tr(\mathbf{PW}^2\mathbf{\Sigma D}^{-2}),
\end{align*} 
and
\begin{align*}
    \eta_k =\frac{\alpha\gamma\beta_k}{1+2\alpha\gamma\,\smax(\mathbf{T}) \beta_k} = \frac{\alpha\gamma\tau h(k)/\left(\alpha\text{Tr}(\mathbf{W\Sigma}))\right)}{1+2\alpha\gamma\,\smax(\mathbf{T})\tau h(k)/\left(\alpha\text{Tr}(\mathbf{W\Sigma}))\right)} \approx \frac{2}{\alpha\gamma k+4\,\smax(\mathbf{T}) }.
\end{align*}
\hfill$\square$\\
It is worth mentioning that all the above results are more general and therefore results from~\cite{marshall2023optimal,tondji2024adaptive} can be obtained from our results.

\section{Lagrange-multiplier derivation of the optimal distribution}\label{app:lag}

\begin{proposition}\label{prop:lag}
For $c_1,\dots,c_m>0$, the minimum of $\sum_i c_i/p_i$ over $\{p:p_i\ge0,\sum_i p_i=1\}$ is
attained at $p_i^{\mathrm{opt}}=\sqrt{c_i}/\sum_j\sqrt{c_j}$ and equals $(\sum_i\sqrt{c_i})^2$.
\end{proposition}
\begin{proof}
The Lagrangian $\mathcal L=\sum_i c_i/p_i-\lambda(\sum_i p_i-1)$ has
$\partial_{p_k}\mathcal L=-c_k/p_k^2-\lambda=0$, so $p_k\propto\sqrt{c_k}$; normalising gives
$p_k^{\mathrm{opt}}=\sqrt{c_k}/\sum_j\sqrt{c_j}$. Substituting back,
$\sum_i c_i/p_i^{\mathrm{opt}}=(\sum_j\sqrt{c_j})\sum_i\sqrt{c_i}=(\sum_i\sqrt{c_i})^2$.
\end{proof}

\end{document}